\documentclass[11pt,a4paper]{amsart}
\usepackage{amssymb,amsmath,amsthm}

\usepackage[colorlinks=true, pdfstartview=FitV, linkcolor=blue, citecolor=blue, urlcolor=blue,pagebackref=false]{hyperref}

\usepackage{esint}
\usepackage{mathrsfs}

\usepackage{enumitem}

\usepackage{times,latexsym,color}

\newcommand{\BOX}{\ensuremath\Box}

\newtheorem{theorem}{Theorem}
\newtheorem{proposition}{Proposition}
\newtheorem{lemma}[proposition]{Lemma}
\newtheorem{corollary}[proposition]{Corollary}

\theoremstyle{remark}
\newtheorem{remark}[proposition]{Remark}

\theoremstyle{definition}
\newtheorem{definition}[proposition]{Definition}

\DeclareMathOperator{\supp}{supp}

\newcommand{\N}{\mathbb{N}}

\newcommand{\R}{\mathbb{R}}
\newcommand{\Z}{\mathbb{Z}}

\newcommand{\ep}{\varepsilon}

\newcommand{\intbar}{{- \hspace{- 1.05 em}} \int}

\definecolor{darkgreen}{rgb}{0,0.5,0}
\definecolor{darkblue}{rgb}{0,0,0.7}
\definecolor{darkred}{rgb}{0.9,0.1,0.1}
\definecolor{lightblue}{rgb}{0,0.51,1}

{\vskip\baselineskip\noindent\textbf{Proof of {#1}:}}%
{\hspace*{.1pt}\hspace*{\fill}\BOX\vskip\baselineskip}
{\vskip\baselineskip\noindent\textbf{Proof of Theorem \protect\ref{#1}:}}%
{\hspace*{.1pt}\hspace*{\fill}\BOX\vskip\baselineskip}
{\vskip\baselineskip\noindent\textbf{Proof of Lemma~\protect\ref{#1}:}}%
{\hspace*{.1pt}\hspace*{\fill}\BOX\vskip\baselineskip}
{\vskip\baselineskip\noindent\textbf{Proof of Proposition~\protect\ref{#1}:}}%
{\hspace*{.1pt}\hspace*{\fill}\BOX\vskip\baselineskip}
{\vskip\baselineskip\noindent\textbf{Proof of Theorems \protect\ref{#1} --
\protect\ref{#2}:}}%
{\hspace*{.1pt}\hspace*{\fill}\BOX\vskip\baselineskip}

\begin{document}

\title[Concentration for blow-up solutions]{Localized smoothing for the Navier-Stokes equations and concentration of critical norms near singularities}

\author[T. Barker]{Tobias Barker}
\address[T. Barker]{DMA, \'{E}cole Normale Sup\'erieure, CNRS, PSL Research University, 75\, 005 Paris}
\email{tobiasbarker5@gmail.com}

\author[C. Prange]{Christophe Prange}
\address[C. Prange]{Universit\'e de Bordeaux, CNRS, UMR [5251], IMB, Bordeaux, France}
\email{christophe.prange@math.u-bordeaux.fr}

\subjclass[2010]{}
\date{\today}

\maketitle

\noindent {\bf Abstract} This paper is concerned with two dual aspects of the regularity question of the Navier-Stokes equations. First, we prove a local in time localized smoothing effect for local energy solutions. More precisely, if the initial data restricted to the unit ball belongs to the scale-critical space $L^3$, then the solution is locally smooth in space for some short time, which is quantified. 
This builds upon the work of Jia and \v{S}ver\'{a}k, who considered the subcritical case. Second, we apply these localized smoothing estimates to prove a concentration phenomenon near a possible Type I blow-up. Namely, we show if $(0, T^*)$ is a singular point then $$\|u(\cdot,t)\|_{L^{3}(B_{R}(0))}\geq \gamma_{univ},\qquad R=O(\sqrt{T^*-t}).$$ This result is inspired by and improves concentration results established by Li, Ozawa, and Wang and  Maekawa,  Miura, and Prange. We also extend our results to other critical spaces, namely $L^{3,\infty}$ and the Besov space $\dot{B}^{-1+\frac3p}_{p,\infty}$, $p\in(3,\infty)$.

\vspace{0.3cm}


\vspace{0.3cm}


\section{Introduction}

{This paper concerns weak Leray-Hopf solutions of the Navier-Stokes system \begin{align}\label{e.nse}
\partial_t u + u \cdot \nabla u - \Delta u + \nabla p =0,\,\,\,\, u(x,0)=u_{0}(x),\,\,\,\, \nabla\cdot u =0, \quad  x\in \R^3, \  t>0.
\end{align}
In particular these solutions satisfy
\begin{equation}\label{globalenergyinequality}
\|u(\cdot,t)\|_{L^{2}(\mathbb{R}^3)}^2+2\int\limits_{0}^{t}\int\limits_{\mathbb{R}^3} |\nabla u(x,s)|^2 dxds\leq \|u_{0}\|_{L^{2}(\mathbb{R}^3)}^{2}
\end{equation}
along with other properties.} 

{Though these solutions were shown to exist for any divergence-free initial data in $L^{2}(\mathbb{R}^3)$ \cite{Leray}, it is unknown if they are smooth for all positive times. Uniqueness is also still open, although non-uniqueness  scenarios were suggested by \cite{JiaSvernonunique} with supporting numerical evidence in \cite{guillod2017numerical}. For weaker notions of solutions with bounded kinetic energy, uniqueness can fail as demonstrated by Buckmaster and Vicol in \cite{buckmaster2017nonuniqueness}. In investigating the regularity of such solutions, it is natural to ask the following question:
\begin{itemize}
\item[] \textbf{(Q)} What type of initial conditions induce smoothing of the associated solutions of the Navier-Stokes equations and can this be described quantatively?
\end{itemize}}
The list of contributions to \textbf{(Q)} is vast and we do not attempt to be exhaustive. The first contribution to \textbf{(Q)} was provided by Leray in \cite{Leray} for $u_0\in L^{p}(\mathbb{R}^3)$ with $p>3$. By using perturbation methods, further contributions to \textbf{(Q)} were made by Kato \cite{Kato} for $u_0\in L^{3}(\mathbb{R}^3)$, by Planchon \cite{Plan96} for $u_0$ in critical Besov spaces, by Koch and Tataru \cite{KochTataru} for $u_0\in BMO^{-1}(\mathbb{R}^3)$, by Maekawa and  Terasawa \cite{maekawa2006} for $u_0\in L^{3}_{uloc}(\mathbb{R}^3)$ and  by Maekawa, Miura and Prange \cite{MMP17a} for $u_0\in L^{3}_{uloc}(\mathbb{R}^3_{+})$. Here 
\begin{align}
L^{p}_{uloc}(\mathbb{R}^3):=\{u_{0}\in L^{p}_{loc}(\mathbb{R}^3):\,\,\,\,\|u_{0}\|_{L^{p}_{uloc}(\mathbb{R}^3)}:= \sup_{x_0\in\mathbb{R}^3} \|u_{0}\|_{L^{p}(B(x_{0}, 1)}<\infty\}.
\end{align}
Recently in \cite{JS14}, Jia and \v{S}ver\'{a}k made an interesting contribution to \textbf{(Q)} when $u_0\in L^{2}_{uloc}$ and $u_0\in L^{m}(B(0,2))$ with $m>3$.
Their result and our first theorem below are in the context of local energy solutions, which were introduced by Lemari\'{e}-Rieusset (see Chapters 32-33 in \cite{LR02}, see also  \cite{KS07}). The definition of local energy solutions is given in Appendix \ref{app.a}. 
Our foremost result is the following theorem.

\begin{theorem}\label{theo.main}
For all $M\in(0,\infty)$, there exists $S^*(M)\in(0,\frac14]$ and an independent universal constant $\gamma_{univ}$ such that the following holds true. Consider any local energy solution $u$, in the sense of Definition \ref{def.lews},  to the Navier-Stokes equations \eqref{e.nse} with initial data $u_0\in L^2_{uloc,\sigma}(\R^3)$ satisfying 
\begin{equation*}
\|u_0\|_{L^2_{uloc}(\R^3)}\leq M\quad\mbox{and}\quad \sup_{|\bar x|\geq R}\|u_0\|_{L^2(B_1(\bar x))}\stackrel{R\rightarrow\infty}{\longrightarrow} 0,
\end{equation*}
\begin{equation*}
u_{0}\in L^{3}(B_{2}(0))\,\,\,\,\,\,\mbox{and}\,\,\,\,\,\,\|u_0\|_{L^3(B_2(0))}\leq \gamma_{univ}.
\end{equation*}
Then the above assumptions imply that
\begin{equation*}
u\in L^\infty(B_\frac13(0)\times(\beta,S^*(M))),
\end{equation*}
for all $\beta\in (0,S^*(M))$.
\end{theorem}

We also prove an extension of this theorem to: (i) the critical Lorentz-space case in Section \ref{app.b}, i.e. $u_0\in L^{3,\infty}(B_2(0))$, and (ii) to the critical Besov space case in Section \ref{app.c}, i.e. $u_0\in \dot{B}^{-1+\frac3p}_{p,\infty}(B_2(0))$, $p\in(3,\infty)$. The case of the Besov case in particular requires some technical innovations, which are outlined at the end of the paragraph \ref{subsec.novelty} below.

This result asserts that the regularity of local energy solutions is a somewhat local pro\-perty, near initial time. Indeed the solution $u$ is bounded in $B_\frac12(0)\times(0,S^*(M))$, hence smooth in space, if the initial data is locally in the scale critical space $L^3$.

Notice that the result we prove in Section \ref{sec.proof} is stronger. Indeed, considering the mild solution $a$ associated to an $L^{3}$ continuous divergence-free extension of the critical data $u_{0}|_{B_{1}(0)}$, we prove that\footnote{In this paper, the parabolic H\"older semi-norm is defined in the following way: $$[\cdot]_{C^{0,\nu}_{par}(\R^3\times [0,T])}:=[\cdot]_{C^{0,\nu}_{t}([0,T];L^\infty_x(\R^3))}+[\cdot]_{L^\infty_t(0,T;C^{0,2\nu}_{x}(\R^3))}.$$}
\begin{equation}\label{e.holdereu-a}
u-a\in C^{0,\nu}_{par}(\bar{B}_\frac13(0)\times[0,S^*(M)]),
\end{equation}
for some $\nu\in (0,\frac12)$. 
We refer to Theorem \ref{localboundednesscriticalStokes} in Section \ref{sec.boundedness} 
and to Section \ref{sec.proof} for the details, in particular regarding the decomposition of the initial data. This improved regularity for $u-a$ relies heavily on the fact that $u-a$ has zero initial data in $B_1(0)$. 

\begin{subsection}{Comparison of Theorem \ref{theo.main} to previous literature and novelty of our results}\label{subsec.novelty}
Our theorem can be seen as an extension to the scale-critical case of the pioneering result of Jia and \v{S}ver\'{a}k \cite[Theorem 3.1]{JS14} for subcritical $u_0\in L^m(B(0,2))$, $m>3$. The main part of their proof relies upon establishing an an $\ep$-regularity criteria for suitable solutions (see Section \ref{sec.morrey} Definition \ref{def.lsws} for a definition of suitable solutions) of the perturbed Navier-Stokes equations with subcritical $a$, i.e. $a\in L^{\infty}_{t}L^{m}_{x}(B_{1}\times (-1,0))$, $m>3$ and $\nabla\cdot a=0$:
\begin{equation}\label{pertubedequation}
\partial_{t}v-\Delta v+v\cdot\nabla v+ a\cdot\nabla v+v\cdot\nabla a+ \nabla q=0,\,\,\,\,\nabla\cdot v=0\,\,\,\,\,\textrm{in}\,\,\, B_{1}(0)\times (-1,0).
\end{equation}
In particular, they show that if certain scale-critical quantities involving $v$ and $q$ on the unit cube $B_{1}(0) \times (-1,0)$ are small then one has decay of the oscillation:
$$
\frac{1}{r^{5}}\int\limits_{t-r^2}^{t_0}\int\limits_{B_{r}(\bar x)} \Big|v-\int\limits_{t-r^2}^{t}\int\limits_{B_{r}(\bar x)} v dyds\Big|^3 dxds'\leq Cr^{\alpha}$$
for all $(\bar x,t)\in B_\frac12\times(-\frac14,0)$ and for some $\alpha>0$. This implies parabolic H\"{o}lder continuity of $v$ by Campanato's characterisation.
The proof of the decay of the oscillation in Jia and \v{S}ver\'{a}k's paper \cite{JS14}  is achieved by contradiction and by compactness arguments. Related arguments were previously used in the context of the Navier-Stokes equations  by Lin in \cite{Lin} and by Ladyzhenskaya and Seregin in \cite{LadySer}. Such arguments applied to the system \eqref{pertubedequation} crucially use that for $a$ in subcritical spaces, 
we have  parabolic H\"{o}lder continuity in $B_{\frac{1}{2}}(0)\times (-\frac{1}{4},0)$ for the following linear system: 
\begin{equation}\label{linearizedperturbed}
\partial_{t} w-\Delta w+ a\cdot\nabla w+w\cdot\nabla a+\nabla q= \nabla\cdot G,\,\,\,\,\,\,\,\nabla\cdot a=0,\,\,\,\,\,\,\,\nabla\cdot w=0\,\,\,\,\,\textrm{in}\,\,\,\,\,B_{1}(0)\times (-1,0),
\end{equation}
for a subcritical forcing term $G\in L^\frac{5m}3(B_{1}(0) \times (-1,0))$. 
Unfortunately, when $u_0$ is critical and hence $a,\, G$ belong to scale invariant spaces with respect to the Navier-Stokes scaling\footnote{The Navier-Stokes equations are invariant under the scaling $(u^{(\lambda)}(x,t), p^{(\lambda)}(x,t))= (\lambda u(\lambda x, \lambda^2 t), \lambda^2 p(\lambda x, \lambda^2 t))$, $u_{0}^{(\lambda)}(x)= \lambda u_{0}(\lambda x)$. We say that a space $X\subset \mathcal{S}^{'}(\mathbb{R}^3)$ is critical (or scale-invariant) if its norm is invariant under the above rescaling for the initial data. Likewise, we say $X_{T}\subset \mathcal{S}^{'}(\mathbb{R}^3)$ is critical is its norm is invariant under the rescaling for the velocity field.} 
such as $L^{5}(B_{1}(0)\times (-1,0))$, we do not expect solutions of \eqref{linearizedperturbed} to be H\"{o}lder continuous. Concerning this point, let us emphasize that the regularity result \eqref{e.holdereu-a} uses in an essential way that $u-a$ has zero initial data locally in $B_1(0)$. This lack of improvement for the perturbed linear system \eqref{linearizedperturbed} seems to prevent us from relying on compactness arguments to directly prove the boundedness. Indeed, such compactness arguments are based on H\"older continuity for the linear system. Difficulties with using compactness arguments  are also found when proving $\ep$- regularity statements for the Navier-Stokes equations in higher dimensions (see \cite{dong2014partial}-\cite{dong2014boundary}). This is the main difficulty we have to overcome to prove Theorem~\ref{theo.main}. We handle this difficulty by proving a subcritical Morrey bound thanks to a Caffarelli, Kohn and Nirenberg-type scheme. This point is explained in more details in the paragraph \ref{subsec.strat} below.
 
The extension of our results to the Besov case in Section \ref{app.c} relies on some ideas which are new as far as we know. In particular, in the Caffarelli, Kohn and Nirenberg-type iteration, we need to exploit the local decay of the kinetic energy near the initial time, because the critical drift is more singular in the Besov case than in the $L^3$ case. Such an insight was used before for \emph{global} estimates by Barker \cite{Barker18} to prove weak-strong uniqueness, in Barker, Seregin and \v{S}ver\'{a}k's paper \cite{BSS18} on global $L^{3,\infty}$ solutions and by Albritton and Barker \cite{AlbrittonBarkerBesov2018} for global Besov solutions. However, to the best of our knowledge, it is the first time that the decay of the kinetic energy near initial time is used in \emph{local} estimates, such as a Caffarelli, Kohn and Nirenberg-type iteration. We believe this point is of independent interest.

\subsection{Strategy of Proof of Theorem \ref{theo.main}}
\label{subsec.strat}
 
As is the case in Jia and \v{S}ver\'{a}k's paper, the key point is to take advantage of the smallness of the local energy of the perturbation $v$ in the unit ball near initial time, i.e. in $B_1(0)\times(0,S^*(M))$. There are then two main blocks in the proof. 

First we prove a subcritical Morrey bound on the perturbation. Smallness of the local energy together with the smallness of $\|a\|_{L^5_{t,x}}$ enables to prove a subcritical Morrey bound on $v$: for $\delta\in(0,3)$ fixed, for $(\bar x,t)\in B_\frac12(0)\times(0,S^*(M))$,
\begin{equation}\label{e.submorrey}
\sup_{r\in(0,\frac12)}\frac1{r^{5-\delta}}\int\limits_{t-r^2}^t\int\limits_{B_r(\bar x)}|v|^3dxds<\infty,
\end{equation}
with $v$ extended by $0$ in negative times. The precise statement is given in Theorem \ref{theo.morrey} in Section \ref{sec.morrey}. 
Estimate \eqref{e.submorrey} is based on a Caffarelli, Kohn, Nirenberg type iteration. The proof requires some technical innovations, in particular concerning the treatment of the pressure and of the perturbation terms $a\cdot\nabla v$ and $v\cdot\nabla a$. A major difficulty is that the decay of $\|a\|_{L^5(B_r(\bar x)\times(t-r^2,t))}$ does not improve when $r\rightarrow 0$. A careful study of the proof in Section \ref{sec.morrey} shows that the $L^5$ norm for $a$ is the critical threshold for iteration to work. Exploiting that $a$ is a solution to the Navier-Stokes equations and the bound \eqref{e.submorrey}, one could directly apply $\ep$-regularity away from the initial time to get smoothness of the perturbation $v$. Instead, we aim at obtaining the boundedness of $v$ up to the initial time. 

Obtaining the boundedness of $u-a$ and eventually H\"older continuity in the parabolic metric up to initial time is the second main block of the paper. It goes through the use of the Morrey bound \eqref{e.submorrey} to control the nonlinear term in \eqref{pertubedequation} and a bootstrap on the linear equation to get the boundedness. Related arguments were used by Seregin in \cite{Ser06}. This work is done in Section \ref{sec.boundedness}.

Let us point out that for the subcritical case $u_0\in L^{m}(B(0,2))$ (which corresponds to $a$ belonging to subcritical spaces), Jia and \v{S}ver\'{a}k prove in \cite{JS14} that the perturbation $v= u-a$ is H\"{o}lder continuous in the parabolic metric up to the initial time. Moreover, in \cite{JS14} the H\"{o}lder exponent degenerates as $m$ approaches the critical case $m=3$. Perhaps at first sight it appears somewhat unexpected that one still obtains H\"{o}lder continuity of $v$ up to the initial time, for the critical initial data case. Our proof for showing this relies upon the structure of estimates for the mild solution $a$, in particular $\sup_{s\in (0,S^*(M))}s^\frac15\|a(\cdot,s)\|_{L^5}\ll 1$ and the fact that $v$ has zero initial data. Such points allow us to obtain a decay of the $L^{\infty}$ norm of $v$ near the initial time, which is key from going from $v$ being bounded to H\"{o}lder continuous.

\end{subsection}

\begin{subsection}{ Concentration of norms centered on singularities}
 In the the second part of the paper, we apply the results of Theorem 1 to obtain certain new concentration results for weak Leray-Hopf solutions which first develop singular points at time $T^{*}>0$. We say that $(\bar x,t)\in\R^3\times(0,\infty)$ is a regular point of $u$, if there exists $r\in(0,\infty)$ such that $u\in L^\infty(B_r(\bar x)\times (t-r^2,t))$. A contrario, a point $(\bar x,t)\in\R^3\times(0,\infty)$ is a singular point, or a blow-up point if it is not regular. A time $T^{*}\in (0,\infty)$ is called a blow-up time if there exists $\bar x\in\R^3$ such that $(\bar x,T^*)$ is a singular point. For $(\bar x,t)\in \R^3\times\R$, we define the parabolic cylinder
\begin{equation*}
Q_r(\bar x,t):=B_r(x)\times(t-r^2,t).
\end{equation*}
 
Investigation of singular weak Leray-Hopf solutions was first performed by Leray in \cite{Leray}. In particular, Leray showed that if a weak Leray-Hopf solution $v$ first develops singularities at $T^{*}$ then
 \begin{equation}\label{Leraynecessary}
 \|v(\cdot,t)\|_{L^{p}(\mathbb{R}^3)}\geq \frac{C(p)}{(T^{*}-t)^{\frac{1}{2}(1-\frac{3}{p})}}\,\,\,\,\textrm{for}\,\,\,\,3<p\leq\infty\,\,\,\,\,\textrm{and}\,\,\,\,0<t<T^{*}.
 \end{equation}
 Behaviour of the $L^{3}$ norm is more subtle. In a breakthrough paper, Escauriaza, Seregin and Sver\'{a}k showed that if $(\bar x, T^{*})$ is a singular point then
 \begin{equation}\label{ESSnecessary}
 \lim\sup_{t\uparrow T^{*}} \|v(\cdot,t)\|_{L^{3}(B_{\delta}(\bar x))}=\infty\,\,\,\,\,\,\textrm{for\,\,\,any\,\,\,fixed}\,\,\,\delta>0.
 \end{equation}
See  \cite{phuc2015navier} and \cite{DW18} for local extensions, as well as \cite{GKP} and \cite{LiWang2018blowup} for global extensions. 
 Later in \cite{seregin2012certain}, Seregin improved (\ref{ESSnecessary}):
 \begin{equation}\label{L3tendtoinfinity}
 \lim_{t\uparrow T^{*}}\|v(\cdot, t)\|_{L^{3}(\mathbb{R}^3)}=\infty.
 \end{equation}
We also refer to \cite{BS17}, \cite{MMP17a}, \cite{Albritton18} and \cite{AlbrittonBarkerBesov2018} for extensions and refinements (we mention that \cite{BS17} and \cite{MMP17a} concern the half-space).   Recently in \cite{AlbrittonBarker2018local} Albritton  and Barker refined \eqref{ESSnecessary} and \eqref{L3tendtoinfinity} to show that if $\Omega$ is a bounded domain with $C^{2}$ boundary one has
 \begin{equation}\label{aalbrittonbarker}
 \lim_{t\uparrow T^{*}} \|v(\cdot,t)\|_{L^{3}(B_{\delta}(\bar x)\cap \Omega)}=\infty\,\,\,\,\,\,\textrm{for\,\,\,any\,\,\,fixed}\,\,\,\delta>0.
 \end{equation} 
 
In this paper we are interested in investigating accumulation behaviour of norms of $v$ near blow-up times $T^{*}$ on balls whose radius shrinks to zero as $t$ approaches $T^{*}$. We refer to this as \textit{`concentration of $v$'}. Such phenomenon was investigated for other equations e.g. nonlinear Schr\"odinger in the wake of the pioneering work \cite{MT90}, see \cite{HK06}, \cite{HR07}. In cite \cite{LOW18}, an interesting concentration result is proven for a weak Leray-Hopf solution $v$ which first blows-up at $T^{*}>0$. In particular, their results imply  that there exists $t_{n}\uparrow T^{*}$ and $x_{n}\in\mathbb{R}^3$ such that
 \begin{equation}\label{Liwangconcentration}
 \|v(\cdot, t_{n})\|_{L^{m}(B_{\sqrt{ C(m) (T^{*}-t_n)}}(x_{n}))}\geq \frac{C(m)}{({T^{*}-t_n})^{\frac{1}{2}(1-\frac{3}{m})}},\,\,\,\,\ 3\leq m\leq \infty.
 \end{equation}
We are not aware of any prior such results of this type for the Navier-Stokes equations.
 
By using a rescaling argument and an estimate of the existence time of mild solutions in terms of the size of the initial data in $L^m_{uloc}(\R^3)$, $m>3$, Maekawa Miura Prange improved (\ref{Liwangconcentration}). In particular, see \cite[Corollary 1.1]{MMP17a}, they showed that for every $t\in (0, T^*)$ (not just a sequence $t_{n}\uparrow T)$ there exists $x(t)\in\mathbb{R}^3$ such that
\begin{equation}\label{maekawamiuraprange}
\|v(\cdot, t)\|_{L^{p}(B_{\sqrt{ C(m) (T^*-t)}}(x(t)))}\geq \frac{C(m)}{({T^*-t})^{\frac{1}{2}(1-\frac{3}{m})}}\,\,\,\,\ 3\leq m\leq \infty.
\end{equation}
In \eqref{Liwangconcentration} or \eqref{maekawamiuraprange} no information is provided on $x_{n}$ and $x(t)$. It is natural to ask whether the concentration phenomenon occurs on balls $B(x,R)$ with $R=O(\sqrt{T^{*}-t})$ and with $(x,T^{*})$ being a singular point. Our second theorem answers this in the affirmative for the $L^{3}$ for Leray-Hopf solutions which first blow-up at time $T^{*}$ and which satisfy the Type I bound: \begin{align}
\begin{split}\label{e.typeI}
&\sup_{\bar x\in\R^3}\sup_{r\in(0,r_0)}\sup_{T^{*}-r^2<t<T^{*}}r^{-\frac12}\Big(\int\limits_{B_r(\bar x)}|u(x,t)|^2dx\Big)^\frac12\leq M,\\
&\qquad
\mbox{for a fixed radius}\ r_0\in (0,\infty]\ \mbox{and}\ M,\ T^{*}\in (0,\infty).
\end{split}
\end{align}
Let us now state our concentration result.

\begin{theorem}\label{theo.conc}
Let $\gamma_{univ}$, $M\in (0,\infty)$ and $S^*(M)$ given by Theorem \ref{theo.main}. Let $T^{*}\in(0,\infty)$ and $r_0\in (0,\infty]$ be fixed. There exists $t_*(T^*,M,r_0)\in[0,\infty)$ such that the following holds true. Let $u$ be a Leray-Hopf solution to \eqref{e.nse} in $\R^3\times(0,\infty)$ satisfying the type I bound \eqref{e.typeI}. Furthermore, suppose $u$ first blows-up at $T^{*}$ and has a singular point at the space-time point $(0,T^{*})$. Then the above assumptions imply that
\begin{equation}\label{e.conblup}
\|u(\cdot,t)\|_{L^3\big(|\cdot|\leq 2\sqrt{\frac{T^{*}-t}{S^*(M)}}\big)}>\gamma_{univ},
\end{equation}
for all $t\in (t_*(M,r_0),T^{*})$. 
\end{theorem}

We further extend this result. Indeed, we also prove: (i) in Section \ref{app.b} the concentration of the critical $L^{3,\infty}$ norm, and (ii) in Section \ref{app.c} the concentration of the critical Besov space norm $\dot{B}^{-1+\frac3p}_{p,\infty}$, $p\in(3,\infty)$.

By translation invariance of the Navier-Stokes equations and of the type I condition \eqref{e.typeI}, the concentration result \eqref{e.conblup} holds at any blow-up point $(\bar x,T^{*})$. Notice also that if $r_0=\infty$, $t_*=0$.

Let $M'\in(0,\infty)$. It is clear that the type I condition \eqref{e.typeI} is satisfied by Leray-Hopf solutions blowing-up at time $T^{*}>0$ and such that
\begin{equation*}
|u(x,t)|\leq \frac{M'}{|x|},\quad\mbox{for all}\ (x,t)\in\R^3\times(0,T^{*}).
\end{equation*}
More generally, it is also satisfied for Leray-Hopf solutions $u$ blowing-up at time $T^{*}>0$ and satisfying a scale-critical Morrey-type bound, i.e. 
\begin{equation*}
\|u(\cdot,t)\|_{\dot M^{2,3}}:=\sup_{\bar x\in\R^3}\sup_{r\in(0,\infty)}r^{-\frac12}\Big(\int\limits_{B_r(0)}|u(x,t)|^2dx\Big)^\frac12\leq M'
\end{equation*}
for all $t\in (0,T^{*})$. This condition corresponds to \eqref{e.typeI} with $r_0=\infty$ and $M=M'$. Hence the concentration in Theorem \ref{theo.conc} holds for any $t\in(0,T^{*})$. It is less obvious to see that type I blow-ups satisfying the bound
\begin{equation}\label{e.typeIa}
\sqrt{T^{*}-t}|u(x,t)|\leq M'
\end{equation}
or
\begin{equation}\label{e.typeIb}
\sqrt{T^{*}-t}^\theta|x|^{1-\theta}|u(x,t)|\leq M'
\end{equation}
for some $\theta\in (0,1)$ also enter the framework of Theorem \ref{theo.conc}. Yet, \eqref{e.typeIa} and \eqref{e.typeIb} imply that there exists $r_0\in (0,\infty)$ and $M(M',u_0, r)\in(0,\infty)$ such that \eqref{e.typeI} holds. This is proved in \cite{SerZajac} (see also p.\, 844-849 of \cite{Seregin2018}). We note that for $r_0=\infty$ this implication fails for the case of the half-space with Dirichlet boundary condition. This is demonstrated by Giga in \cite[Theorem 3.1]{Giga13} by using shear flows.

\end{subsection}
\begin{subsection}{Strategy of proof}
The strategy to prove Theorem \ref{theo.conc} is to rescale the solution appropriately and then reduce to Theorem \ref{theo.main}. The Type I condition \eqref{e.typeI} ensures that the rescaled solution has initial data that can be controlled in $L^{2}_{uloc}(\R^3)$. We refer to Section \ref{sec.proof} for more details.
\end{subsection}
\begin{subsection}{Final discussion}
As a corollary to Theorem \ref{theo.main} we see that if $u$ is a weak Leray-Hopf solution (with initial data $u_{0}\in L^{2}(\mathbb{R}^3)$) which first blows up at $T^{*}>0$  and has a singular point  $(0,T^{*})$, then the following holds true. Namely, there exists $t_{*}(T^*,\|u_{0}\|_{L^{2}(\mathbb{R}^3)})$ such that for $t\in [t_{*}, T^{*})$ we have:
\begin{equation}\label{concentrationWLHunitball}
\|u(\cdot, t)\|_{L^{3}(B_{1}(0))}> \gamma_{univ}.
\end{equation} 
Although the result in \cite{AlbrittonBarker2018local} shows that $\lim_{t\uparrow T^{*}}\|u(\cdot,t)\|_{L^{3}(B_{1}(0))}=\infty$, it does not provide quantitative information on at which moment in time $\|u(\cdot,t)\|_{L^{3}(B_{1}(0))}$ begins to grow.

Soon after the present work was submitted to arXiv, Kang, Miura and Tsai uploaded to arXiv an independent work \cite{KMT18} with a different proof of Theorem \ref{theo.main}. Their proof of the  subcritical Morrey bounds (Theorem \ref{theo.morrey}) is completely different to ours and relies upon compactness arguments as opposed to a Caffarelli, Kohn and Nirenberg type iteration. The H\"older continuity of Theorem \ref{localboundednesscriticalStokes} and the extension of Theorem \ref{theo.main} to wider critical spaces is not present in \cite{KMT18}. 

\end{subsection}

\subsection*{Outline of the paper}

Section \ref{sec.morrey} is devoted to the proof of a Morrey bound such as \eqref{e.submorrey}. The main result in this section is Theorem \ref{theo.morrey}, which is proved using a Caffarelli, Kohn and Nirenberg type iteration. Section \ref{sec.boundedness} handles the bootstrap arguments on the perturbed linear system in order to prove the boundedness and the H\"older continuity of the perturbation $u-a$ up to initial time. The main result in this section is Theorem \ref{localboundednesscriticalStokes}. Section \ref{sec.proof} is concerned with the proof of Theorem \ref{theo.main} and its application to the concentration of the $L^3$ norm near a potential singularity, Theorem \ref{theo.conc}. In Appendix \ref{app.a} we recall well-known results about mild solutions and local energy solutions, and we give pressure formulas. Appendix \ref{app.b} is devoted to the extension of our results to the Lorentz space $L^{3,\infty}$. Appendix \ref{app.c} is concerned with the extension of our results to the Besov space $\dot{B}^{-1+\frac3p}_{p,\infty}$ for $p\in (3,\infty)$, which requires some new ideas.

Throughout the paper the constant $C\in (0,\infty)$ denotes a universal constant, unless stated otherwise.

\section{Propagation of a Morrey-type bound}
\label{sec.morrey}

The goal of this section is to prove a Morrey-type bound for local suitable solutions of
\begin{align}\label{e.pertns}
\partial_tv -\Delta v+\nabla q\ =-v\cdot\nabla v-a\cdot\nabla v-\nabla\cdot(a\otimes v),\qquad \nabla\cdot v\ =0,
\end{align}
by using a Caffarelli, Kohn and Nirenberg \cite{CKN82} type iteration.

\begin{definition}[Local suitable solution]\label{def.lsws}
Let $a\in L^5(Q_1)$. The pair $(v,q)$ is a local suitable solution in $Q_1$ to the perturbed Navier-Stokes equations \eqref{e.pertns} if  
\begin{align*}
v\ \mbox{belongs to}\ L^\infty(-1,0;L^2(B_1(0)))\cap L^2(-1,0;H^1(B_1(0))),\quad q\ \mbox{belongs to}\ L^\frac32(Q_1),
\end{align*}
$v$ is in $C_{w}([-1,0]; L^{2}(B_{1}(0)))$, $(v,q)$ is a solution to \eqref{e.pertns} in $Q_1$ in the sense of distributions and for all $0\leq \phi\in C^\infty_c(Q_1)$, we have the following local energy equality
\begin{align}\label{e.lei}
\begin{split}
&\int\limits_{B_1(0)}|v(x,t)|^2\phi(x,t) dx+2\int\limits_{-1}^t\int\limits_{B_1(0)}|\nabla v|^2\phi dxds\\
\leq\ &\int\limits_{-1}^t\int\limits_{B_1(0)}|v|^2(\partial_t\phi+\Delta\phi)dxds+\int\limits_{-1}^t\int\limits_{B_1(0)}(|v|^2+2q)v\cdot\nabla\phi dxds\\
&-\int\limits_{-1}^t\int\limits_{B_1(0)}(a\cdot\nabla v)\cdot v\phi dxds+\int\limits_{-1}^t\int\limits_{B_1(0)}(a\otimes v):(\nabla v \phi +v\otimes\nabla\phi)dxds
\end{split}
\end{align}
for all $t\in(-1,0]$.
\end{definition}

The following theorem is a generalization to scale-critical drifts of the $\ep$-regularity result for subcritical drifts proved in the paper by Jia and \v{S}ver\'{a}k \cite{JS14}.

\begin{theorem}\label{theo.morrey}
For all $\delta\in(0,3)$, 
, there exists $C_*(\delta)\in (0,\infty)$, for all $E\in(0,\infty)$, 
there exists $\ep_*(\delta,E)\in (0,\infty)$, for all $a\in L^5(Q_1)$ and all local suitable solution $v$ to \eqref{e.pertns} in $Q_1(0,0)$ such that
\begin{equation}\label{e.locenE}
\sup_{-1<s<0}\int\limits_{B_1(0)}|v(x,s)|^2dx+\int\limits_{Q_{1}(0,0)}|\nabla v|^2dxds\leq E,
\end{equation}
the conditions
\begin{equation}\label{e.epmorreya}
\|a\|_{L^5(Q_1(0,0))}
\leq \ep_*
\end{equation}
and
\begin{equation}\label{e.epmorrey}
\int\limits_{Q_1(0,0)}|v|^3+|q|^\frac32dxds\leq\ep_*
\end{equation}
imply that for all $(\bar x,t)\in \bar Q_{1/2}(0,0)$, for all $r\in(0,\frac14]$,
\begin{equation}\label{e.morreybound}
\intbar_{Q_r(\bar x,t)}|v|^3dxds\leq C_*\ep_*^\frac23r^{-\delta}.
\end{equation}
\end{theorem}

Let $(\bar x,t)\in \bar Q_{1/2}(0,0)$ be fixed for the rest of this section. For all $n\in\N$, we let $r_n:=2^{-n}$. We actually prove that for all $n\geq 2$,
\begin{equation}\label{e.morryrn}
\intbar_{Q_{r_n}(\bar x,t)}|v|^3dxds\leq \ep_*^\frac23r_n^{-\delta}.
\end{equation}
In this estimate, contrary to \eqref{e.morreybound}, there is no constant $C_*(\delta)$ in the right hand side. The constant $C_*(\delta)$ comes from going from scales $r_n$ to all $r\in(0,\frac14]$.  

The proof is by iteration following the scheme of Caffarelli, Kohn and Nirenberg \cite{CKN82} (see also \cite{RRS16}). Our aim is to propagate for $k\geq 2$ the following two bounds
\begin{align}
\tag{$A_k$}\label{e.ak} \frac1{r_k^2}\int\limits_{Q_{r_k}(\bar x,t)}|v|^3dxds+\frac1{r_k^{\frac{1+\delta}2}}\int\limits_{Q_{r_k(\bar x,t)}}|q-(q)_{r_k}(s)|^\frac32dxds\leq\ &\ep_*^\frac23r_k^{3-\delta},\\
\tag{$B_k$}\label{e.bk} \sup_{t-r_k^2<s<t}\int\limits_{B_{r_k}(\bar x)}|v(x,s)|^2dx+\int\limits_{Q_{r_k}(\bar x,t)}|\nabla v|^2dxds\leq\ & C_B\ep_*^\frac23r_k^{3-\frac23\delta}
\end{align}
where 
\begin{equation*}
(q)_{r_k}(s):=\intbar_{B_{r_k}(\bar x)}q(x,s)dx,
\end{equation*}
for constants $\ep_*(\delta,E),\, C_B(\delta)\in (0,\infty)$ to be chosen such that
\begin{equation}\label{e.choiceCB}
C_B=10C_1^2\big(\tfrac{2^{11}}{1-2^{-\frac{2}3\delta}}+2^6\big)
\end{equation}
and
\begin{multline}\label{e.choiceepstar}
\max\big(D_1(\delta)\ep_*^\frac13,D_{2}(1+E)\ep_*^\frac13,5D_3(\delta)\ep_*,D_4(\delta)\ep_*^\frac79,2C_3C_B^\frac32\ep_*^\frac13,\\
12D_5(\delta)\ep_*^\frac13,12D_6(\delta)\ep_*^\frac{2}{15},12D_7(\delta)\ep_*^\frac13,12D_8(\delta)\ep_*^\frac43\big)\leq 1,
\end{multline}
where the constants $D_1,\ldots\,$ are defined in the course of the proof of Theorem \ref{theo.morrey}. 
Notice that \eqref{e.ak} is a bound on a scale-invariant quantity for $v$. However the quantity on the left hand side of \eqref{e.ak} related to the pressure $q$ is not scale-critical. Indeed, we allow for some room in the rate of decay in $r_k$ for the oscillation of the pressure, which gives more flexibility in the argument. The bound \eqref{e.bk} is a bound on the local energy.

\begin{remark}
The reason one cannot take $\delta=0$ appears clearly in the proof of Theorem \ref{theo.morrey}. Indeed, the lack of improvement of decay of $\|a\|_{L^5(Q_r)}$ combined with $\delta=0$ would lead to a linear growth in $k$ when controlling some terms, in particular $I_3$ and $I_4$ below. We would then not be able to offset this growth by taking $\ep_*$ small uniformly in $k$. Notice that if one knows for some reason that $a$ has more integrability, then one can then take $\delta=0$. This would be the case for instance if $a$ solves the Navier-Stokes equations so that one can apply a Serrin's type criteria. However, in view of proving Theorem \ref{theo.main}, we need Theorem \ref{theo.morrey} for a general $a$ small in $L^5_{t,x}$, not necessarily a solution to the Navier-Stokes equations.
\end{remark}

The proof of \eqref{e.bk} for $k=n+1\geq 3$ relies on the local energy inequality \eqref{e.lei}, the bounds \eqref{e.ak} and \eqref{e.bk} for all $2\leq k\leq n$. The idea is to use test functions $(\phi_n)_{n\geq 2}$, which are almost solutions to the backward heat equation. The following lemma is taken directly from \cite[Lemma 15.11]{RRS16}, see also \cite{CKN82}. We give a statement for the sake of completeness.

\begin{lemma}[Construction of test functions $\phi_n$]\label{lem.test}
There exists a constant $C_1\in(0,\infty)$ such that for all $(\bar x,t)\in\R^3\times (0,\infty)$, for all $n\in\N$, $n\geq 2$, there exists 
\begin{equation*}
\phi_n\in C^\infty_c\big(B_{\frac12}(\bar x)\times(t-\tfrac19,t+\tfrac{r_n^2}{2})\big)
\end{equation*}
such that
\begin{align}
&C_1^{-1}r_n^{-1}\leq\phi_n\leq C_1r_n^{-1}\quad\mbox{and}\quad |\nabla\phi_n|\leq C_1r_n^{-2}\ \mbox{on}\quad Q_{r_n}(\bar x,t),\\
&\phi_n\leq C_1r_n^2r_k^{-3}\quad\mbox{and}\quad|\nabla\phi_n|\leq C_1r_n^2r_k^{-4}\ \mbox{on}\ Q_{r_{k-1}}(\bar x,t)\setminus Q_{r_k}(\bar x,t),\ 2\leq k\leq n,\\
&(\supp\phi_n)\cap Q_1(\bar x,t)\subset Q_{\frac13}(\bar x,t),\\
&|\partial_t\phi_n+\Delta\phi_n|\leq C_1r_n^2\quad\mbox{on}\quad \R^3\times(-\infty,0],
\end{align}
where $r_n=2^{-n}$.
\end{lemma}

Assuming \eqref{e.bk} for $2\leq k\leq n$, we easily get the bound \eqref{e.ak} for $k=n$ on $v$ by interpolation. To prove the pressure bound, we need a representation formula for the pressure.

\begin{lemma}[Pressure estimate]\label{lem.pressure}
There exists a constant $C_2\in(0,\infty)$ such that for all $\rho\in(0,\infty)$, for all $a\in L^5(Q_\rho(0,0))$, for all weak solution $q\in L^\frac32(Q_\rho(0,0))$ to 
\begin{equation*}
-\Delta q=\nabla\cdot\nabla\cdot(v\otimes v)+\nabla\cdot\nabla\cdot(a\otimes v)+\nabla\cdot\nabla\cdot(v\otimes a)\quad\mbox{in}\quad Q_\rho(0,0), 
\end{equation*}
we have 
\begin{align}
\label{e.estpress}
\begin{split}
&r^{-\frac{1+\delta}2}\int\limits_{Q_r(0,0)}|q-(q)_r(s)|^\frac32dxds\\
\leq\ & C_2r^{-\frac{1+\delta}2}\int\limits_{Q_{2r}(0,0)}|v|^3dxds+C_2r^{\frac{1-\delta}2}\Bigg(\int\limits_{Q_{2r}(0,0)}|v|^3dxds\Bigg)^\frac12\Bigg(\int\limits_{Q_{2r}(0,0)}|a|^5dxds\Bigg)^\frac3{10}\\
&+C_2r^{6-\frac\delta2}\Bigg(\sup_{-r^2<s<0}\int\limits_{2r<|x|<\rho}\frac{|v(x,s)|^2}{|x|^4}dx\Bigg)^\frac32\\
&+C_2r^{4-\frac\delta2}\int\limits_{-r^2}^0\Bigg(\int\limits_{2r<|x|<\rho}\frac{|v(x,s)a(x,s)|}{|x|^4}dx\Bigg)^\frac32ds\\
&+C_2r^{4-\frac\delta2}\rho^{-\frac92}\int\limits_{Q_\rho(0,0)}|v|^3+|q|^\frac32dxds\\
&+C_2r^\frac{44-5\delta}{10}\rho^{-\frac{39}{10}}\Bigg(\int\limits_{Q_\rho(0,0)}|v|^3dxds\Bigg)^\frac12\Bigg(\int\limits_{Q_\rho(0,0)}|a|^5dxds\Bigg)^\frac3{10},
\end{split}
\end{align}
for all $0<r\leq\rho/2$.
\end{lemma}

\begin{proof}[Proof of Lemma \ref{lem.pressure}]
The proof follows the lines of \cite[Lemma 15.12]{RRS16}. We first adapt the decomposition of the pressure given in Lemma \ref{pressure localisation}. We take a cut-off function $\varphi$ such that $\varphi=1$ on $B(0,\frac34\rho)$, $\supp\varphi\subset B(0,\rho)$ and $\rho|\nabla\varphi|\leq K_2$ and $\rho^2|\nabla^2\varphi|\leq K_2'$ for universal constants $K_2,\, K_2'\in(0,\infty)$. 
We have for all $x\in\R^3$ and almost every $s\in (-\rho^2,0)$,
\begin{align*}
\varphi q(x,s)=\ &-\nabla^2N\ast(\varphi v\otimes v)-\nabla^2N\ast(\varphi a\otimes v)-\nabla^2N\ast(\varphi v\otimes a)\\
&-2\nabla N\ast(\nabla\varphi\cdot (v\otimes v))-2\nabla N\ast(\nabla\varphi\cdot (a\otimes v))-2\nabla N\ast(\nabla\varphi\cdot (v\otimes a))\\
&-N\ast(\nabla^2\varphi:(v\otimes v))-N\ast(\nabla^2\varphi:(a\otimes v))-N\ast(\nabla^2\varphi:(v\otimes a))\\
&-2\nabla N\ast((\nabla \varphi)q)-N\ast(q\Delta\varphi)\\
=:\ &q_{1,1}(x.s)+q_{1,2}(x.s)+q_{1,3}(x.s)+q_{2,1}(x,s)+q_{2,2}(x,s)+q_{2,3}(x,s)\\
&+\ldots\, +q_4(x,s)+q_5(x,s),
\end{align*}
where $N(x)=-\frac{1}{4\pi |x|}$. 
We focus on the terms which involve $a$. The other terms are estimated exactly as in \cite{RRS16}. For $q_1$ split between a local part and a nonlocal part as follows: for all $x\in\R^3$ and almost every $s\in (-\rho^2,0)$
\begin{align*}
q_{1,1}(x,s)=\ &-\int\limits_{B_{2r}(0)}\nabla^2N(x-y):(a\otimes v)\varphi\, dy-\int\limits_{\R^3\setminus B_{2r}(0)}\nabla^2N(x-y):(a\otimes v)\varphi\, dy\\
=\ &q_{1,1,loc}(x,s)+q_{1,1,nonloc}(x,s).
\end{align*}
For the local part, we use Calder\'on-Zygmund estimates and obtain
\begin{align*}
\int\limits_{Q_r(0,0)}|q_{1,1,loc}-(q_{1,1,loc})_r(s)|^\frac32dxds\leq\ &Cr\Bigg(\int\limits_{Q_r(0,0)}|q_{1,1,loc}-(q_{1,1,loc})_r(s)|^\frac{15}8dxds\Bigg)^\frac45\\
\leq\ &Cr\|a\|_{L^5(Q_1(0,0))}^\frac32\|v\|_{L^3(Q_{2r}(0,0))}^\frac32,
\end{align*}
which yields the second term in the right hand side of \eqref{e.estpress}. As for the nonlocal pressure, we estimate its gradient as follows for $x\in B_r(0)$ and $s\in(-1,0)$
\begin{equation*}
|\nabla q_{1,1,nonloc}(x,s)|\leq C\int\limits_{B_\rho(0)\setminus B_{2r}(0)}\frac{|av|}{|y|^4}dy,
\end{equation*}
so that
\begin{align*}
\|q_{1,1,nonloc}-(q_{1,1,nonloc})_r(s)\|_{L^\frac32(B_r(0))}\leq Cr^3\int\limits_{B_\rho(0)\setminus B_{2r}(0)}\frac{|av|}{|y|^4}dy
\end{align*}
which gives the fourth term in the right hand side of \eqref{e.estpress} by integrating in time. It remains to see how $q_{2,j}$ and $q_{3,j}$, $j=2,\, 3$, lead to the last term in the right hand side of \eqref{e.estpress}. We have
\begin{align*}
|\nabla q_{i,j}(x,s)|\leq\ &C\rho^{-4}\int\limits_{B_{\rho}(0)\setminus B_{\frac\rho2}(0)}|a||v|dx\\
\leq\ &C\rho^{-\frac{13}5}\|a(\cdot,s)\|_{L^5(B_1(0))}\|v(\cdot,s)\|_{L^3(B_\rho(0))},
\end{align*}
for $i,\, j=2,\, 3$. Hence, 
\begin{align*}
\|q_{i,j}-(q_{i,j})_r(s)\|_{L^\frac32(Q_r(0,0))}^\frac32\leq\ &Cr^\frac92\int\limits_{-r^2}^0\|\nabla q_{i,j}(\cdot,s)\|_{L^\infty(B_r(0))}^\frac32ds\\
\leq\ &Cr^\frac92\rho^{-\frac{39}{10}}\int\limits_{-r^2}^0\|a(\cdot,s)\|_{L^5(B_1(0))}^\frac32\|v(\cdot,s)\|_{L^3(B_\rho(0))}^\frac32ds\\
\leq\ &Cr^{\frac{49}{10}}\rho^{-\frac{39}{10}}\|a\|_{L^5(Q_1(0,0))}^\frac32\|v\|_{L^3(Q_\rho(0,0))}^\frac32
\end{align*}
which concludes the proof.
\end{proof}

With these two lemmas, we can now proceed with the proof of Morrey bound.

\begin{proof}[Proof of Theorem \ref{theo.morrey}]
\noindent{\bf \underline{Step 1:} \eqref{e.ak} for $k=2$.} This is a direct consequence of assumption \eqref{e.epmorrey}. Indeed, $r_2=2^{-2}$ so that
\begin{align*}
&\frac1{r_2^2}\int\limits_{Q_{r_2}(\bar x,t)}|v|^3dxds+\frac1{r_2^{\frac{1+\delta}2}}\int\limits_{Q_{r_2}(\bar x,t)}|q-(q)_{r_2}(s)|^\frac32dxds\\
\leq\ &16\int\limits_{Q_{1/2}(\bar x,t)}|v|^3dxds+16\int\limits_{Q_{1/2}(\bar x,t)}|q|^\frac32dxds\\
\leq\ &16\ep_*=2^{7-\delta}\ep_*^\frac13\ep_*^\frac23r_1^{3-\delta}=D_1(\delta)\ep_*^\frac13\ep_*^\frac23r_1^{3-\delta}\leq \ep_*^\frac23r_1^{3-\delta},
\end{align*}
by our choice of $\ep_*$, see \eqref{e.choiceepstar}.

\noindent{\bf \underline{Step 2:} \eqref{e.ak} for $k=2$ and \eqref{e.locenE} implies \eqref{e.bk} for $k=2$.} We take $\phi=\phi_2$ in the local energy inequality \eqref{e.lei} and bound every term in the right hand side. Then
\begin{align*}
&C_1^{-1}r_2^{-1}\sup_{t-r_2^2<s<t}\int\limits_{B_{r_2}(\bar x)}|v(x,s)|^2dx+C_1^{-1}r_2^{-1}\int\limits_{Q_{r_2}(\bar x,t)}|\nabla v|^2dxds\\
\leq\ &C_1r_2^2\int\limits_{Q_{1/2}(\bar x,t)}|v|^2dxds+\int\limits_{Q_{1/2}(\bar x,t)}|v|^3|\nabla\phi_2|dxds+2\int\limits_{Q_{1/2}(\bar x,t)}|v||q||\nabla\phi_2|dxds\\
&+\int\limits_{Q_{1/2}(\bar x,t)}|a||v||\nabla v||\phi_2|dxds+\int\limits_{Q_{1/2}(\bar x,t)}|a||v|(|\nabla v||\phi_2| +|v||\nabla\phi_2|)dxds\\
\leq\ &C_1r_2^22^{-\frac53}\ep_*^\frac23+3C_12^4\ep_*+C_12^3\|a\|_{L^5}E+C_12^{4-\frac23}\|a\|_{L^5}E\\
\leq\ &C_1(1+2^{10}\ep_*^\frac13+2^9E\ep_*^\frac13)\ep_*^\frac23r_2^{2-\frac23\delta}\\
\leq\ &C_1(1+D_{2}(1+E)\ep_*^\frac13)\ep_*^\frac23r_2^{2-\frac23\delta}\leq 2C_1\ep_*^\frac23r_2^{2-\frac23\delta}\leq C_B\ep_*^\frac23r_2^{2-\frac23\delta},
\end{align*}
where the last line follows from the choice of $\ep_*$ (see \eqref{e.choiceCB}) and $C_B$ (see \eqref{e.choiceepstar}).

\smallskip

Let $n\geq 2$.

\noindent{\bf \underline{Step 3:} \eqref{e.ak} and \eqref{e.bk} for all $2\leq k\leq n$ implies \eqref{e.bk} for $k=n+1$.} Let us first notice that assuming \eqref{e.bk} for $2\leq k<n$ in this argument is only needed in the case $a\neq 0$. Step 3 relies on the local energy inequality \eqref{e.lei} and the use of the test function $\phi_n$ constructed in Lemma \ref{lem.test}. We have
\begin{align*}
&C_1^{-1}r_n^{-1}\sup_{t-r_n^2<s<t}\int\limits_{B_{r_n}(\bar x)}|v(x,s)|^2dx+C_1^{-1}r_n^{-1}\int\limits_{Q_{r_n}(\bar x,t)}|\nabla v|^2dxds\\
\leq\ &C_1r_n^2\int\limits_{Q_{1/2}(\bar x,t)}|v|^2dxds+\int\limits_{Q_{1/2}(\bar x,t)}|v|^3|\nabla\phi_n|dxds+2\left|\int\limits_{Q_{1/2}(\bar x,t)}v\cdot\nabla\phi_n qdxds\right|\\
&+2\int\limits_{Q_{1/2}(\bar x,t)}|a||v||\nabla v||\phi_n|dxds+\int\limits_{Q_{1/2}(\bar x,t)}|a||v|^2|\nabla\phi_n|dxds\\
=\ &I_1+\ldots\, I_5.
\end{align*}
It immediately follows from the smallness hypothesis \eqref{e.epmorrey} that
\begin{equation*}
I_1\leq C_1\ep_*^\frac23r_n^{2-\frac23\delta}\leq \frac{C_B}{10C_1}\ep_*^\frac23r_n^{2-\frac23\delta}.
\end{equation*} 
For the other terms, one decomposes $Q_{1/2}(\bar x,t)$ into the union of $Q_{r_n}(\bar x,t)$ and the annuli $Q_{r_{k-1}}(\bar x,t)\setminus Q_{r_k}(\bar x,t)$ for $2\leq k\leq n$. The second and third terms are treated in a standard way. We have using \eqref{e.ak} for $2\leq k\leq n$,
\begin{align*}
I_2\leq\ &C_1r_n^{-2}\int\limits_{Q_{r_n}(\bar x,t)}|v|^3dxds+C_1r_n^2\sum_{k=2}^nr_k^{-4}\int\limits_{Q_{r_{k-1}}(\bar x,t)\setminus Q_{r_k}(\bar x,t)}|v|^3dxds\\
\leq\ &C_1\ep_*^\frac23r_n^{3-\delta}+C_1\ep_*^\frac23r_n^2\sum_{k=3}^nr_k^{-4}r_{k-1}^{5-\delta}+C_1r_n^2r_2^{-4}\int\limits_{Q_{1/2}(\bar x,t)}|v|^3dxds\\
\leq\ &C_1\Big(1+\frac{2^{2(1+\delta)}}{1-2^{-(1-\delta)}}+2^{8}\Big)\ep_*^\frac23r_n^{2-\frac23\delta}\leq  \frac{C_B}{10C_1}\ep_*^\frac23r_n^{2-\frac23\delta}.
\end{align*}
The term $I_3$ requires more care. The idea is to write it as a telescoping series. Indeed, one needs to substract the mean of the pressure on $Q_{r_{k-1}}(\bar x,t)$. In order to do this, we introduce cut-off functions $\chi_k\in C^\infty_c(\R^3\times(-\infty,s))$ for $k\geq 1$ such that $0\leq\chi_k\leq 1$, $\chi_k\equiv 1$ on $Q_{7r_k/8}(\bar x,t)$ and $\supp\chi_k\cap \bar Q_1(\bar x,t)\subset \bar Q_{r_k}(\bar x,t)$. We then have
\begin{align*}
I_3=\ &2\sum_{k=2}^n\int\limits_{Q_{r_{k-1}}}(q-(q)_{r_{k-1}}(s))v\cdot\nabla((\chi_{k-1}-\chi_k)\phi_n)dxds\\
&+2\int\limits_{Q_{r_n}}(q-(q)_{r_n}(s))v\cdot\nabla(\chi_{n}\phi_n)dxds\\
\leq\ &2^{11}C_1r_n^2\sum_{k=2}^nr_{k-1}^{-4}\int\limits_{Q_{r_{k-1}}}|q-(q)_{r_{k-1}}(s)||v|dxds+2^6C_1r_n^{-2}\int\limits_{Q_{r_n}}|q-(q)_{r_n}(s)||v|dxds\\
\leq\ &2^{11}C_1\ep_*^\frac23r_n^2\sum_{k=2}^nr_{k-1}^{-\frac{2}3\delta}+2^6C_1\ep_*^\frac23r_n^{2-\frac{2}3\delta}\\
\leq\ &C_1\Big(\frac{2^{11}}{1-2^{-\frac{2}3\delta}}+2^6\Big)\ep_*^\frac23r_n^{2-\frac{2}3\delta}\leq  \frac{C_B}{10C_1}\ep_*^\frac23r_n^{2-\frac{2}3\delta}.
\end{align*}
In the above calculation we have used the bounds:
$$\|\nabla((\chi_{k}-\chi_{k+1})\phi_{n})\|_{L^{\infty}(Q_{r_n})}\leq 2^{10}C_{1}r_{n}^{2}r_{k}^{-4}\,\,\,\,\,\,\,\textrm{and}\,\,\,\,\,\,\, \|\nabla(\chi_{n}\phi_{n}))\|_{L^{\infty}(Q_{r_{n}})}\leq 17C_{1}r_{n}^{-2}. $$ We refer the reader to p.295 of \cite{RRS16}. As for $I_4$ and $I_5$ we have
\begin{align*}
I_4\leq\ &2C_1r_n^2\sum_{k=2}^nr_k^{-3}\int\limits_{Q_{r_{k-1}}(\bar x,t)\setminus Q_{r_k}(\bar x,t)}|a||v||\nabla v|dxds\\
&+2C_1r_n^{-1}\int\limits_{Q_{r_n}(\bar x,t)
}|a||v||\nabla v|dxds\\
\leq\ &2C_1r_n^2\|a\|_{L^5(Q_{1})}\sum_{k=1}^nr_{k+1}^{-3}\Bigg(\int\limits_{Q_{r_{k}}(\bar x,t)}|v|^\frac{10}3dxds\Bigg)^\frac3{10}\Bigg(\int\limits_{Q_{r_{k}}(\bar x,t)}|\nabla v|^2dxds\Bigg)^\frac12\\
\leq\ &2C_1C_B\ep_*^{\frac23+1}r_n^2\sum_{k=1}^nr_{k+1}^{-3}r_k^{3-\frac23\delta}\\
\leq\ &\frac{2^4}{1-2^{-\frac23\delta}}C_1C_B\ep_*^{\frac23+1}r_n^{2-\frac23\delta}=D_3(\delta)\ep_*C_B\ep_*^{\frac23}r_n^{2-\frac23\delta}\leq  \frac{C_B}{10C_1}\ep_*^{\frac23}r_n^{2-\frac23\delta}
\end{align*}
and
\begin{align*}
I_5\leq\ &C_1r_n^2\|a\|_{L^5(Q_{1}(0,0))}\sum_{k=1}^nr_{k+1}^{-4}r_k^\frac23\Bigg(\int\limits_{Q_{r_{k}}(\bar x,t)}|v|^3dxds\Bigg)^\frac23\\
\leq\ &C_1\ep_*^{1+\frac49}r_n^2\sum_{k=1}^nr_{k+1}^{-4}r_k^{4-\frac23\delta}\\
\leq\ &\frac{2^4}{1-2^{-\frac23\delta}}C_1\ep_*^{1+\frac49}r_n^{2-\frac23\delta}=D_4(\delta)\ep_*^\frac79\ep_*^\frac23r_n^{2-\frac23\delta}\leq  \frac{C_B}{10C_1}\ep_*^\frac23r_n^{2-\frac23\delta}.
\end{align*}
These estimates imply \eqref{e.bk} for $k=n+1$. Indeed, using the lower bound for $\phi_n$ on $Q_{r_n}(\bar x,t)$ and the fact that $Q_{r_{n+1}}(\bar x,t)\subset Q_{r_n}(\bar x,t)$, we obtain
\begin{align*}
&r_{n+1}^{-1}\sup_{t-r_{n+1}^2<s<t}\int\limits_{B_{r_{n+1}}(\bar x)}|v(x,s)|^2dx+r_{n+1}^{-1}\int\limits_{Q_{r_{n+1}}(\bar x,t)}|\nabla v|^2dxds\\
\leq\ &2r_{n}^{-1}\sup_{t-r_{n}^2<s<t}\int\limits_{B_{r_{n}}(\bar x)}|v(x,s)|^2dx+2r_{n}^{-1}\int\limits_{Q_{r_{n}}(\bar x,t)}|\nabla v|^2dxds\\
\leq\ &2C_1(I_1+\ldots\, I_5)\\
\leq\ &C_B\ep_*^\frac23r_{n+1}^{2-\frac23\delta},
\end{align*}
which is the result.

\noindent{\bf \underline{Step 4:} \eqref{e.bk} for $2\leq k\leq n+1$ implies \eqref{e.ak} for $k=n+1$.} First, by interpolation, we ea\-sily get that there exists  a universal constant $C_3\in(0,\infty)$ such that
\begin{align*}
&r_{n+1}^{-2}\int\limits_{Q_{r_n}(\bar x,t)}|v|^3dxds\\
\leq\ & C_3\Bigg(r_{n+1}^{-1}\sup_{t-r_{n+1}^2<s<t}\int\limits_{B_{r_{n+1}}(\bar x)}|v(x,s)|^2dx+r_{n+1}^{-1}\int\limits_{Q_{r_{n+1}}(\bar x,t)}|\nabla v|^2dxds\Bigg)^\frac32.
\end{align*}
Therefore, by \eqref{e.ak} for $k=n+1$,
\begin{equation*}
r_{n+1}^{-2}\int\limits_{Q_{r_{n+1}}(\bar x,t)}|v|^3dxds\leq C_3C_B^\frac32\ep_*r_{n+1}^{3-\delta}
\leq\frac12\ep_*^\frac23r_{n+1}^{3-\delta}.
\end{equation*}
The control of the pressure part is more difficult. We rely on Lemma \ref{lem.pressure}. Hence, taking $r:=r_{n+1}$ and $\rho=\frac14$, we have
\begin{align*}
&r_{n+1}^{-\frac{1+\delta}2}\int\limits_{Q_{r_{n+1}(\bar x,t)}}|q-(q)_{r_{n+1}}(s)|^\frac32dxds\\
\leq\ & C_2r_{n+1}^{-\frac{1+\delta}2}\int\limits_{Q_{{r_{n}(\bar x,t)}}}|v|^3dxds+C_2r_{n+1}^{\frac{1-\delta}2}\Bigg(\int\limits_{Q_{r_{n}(\bar x,t)}}|v|^3dxds\Bigg)^\frac12\Bigg(\int\limits_{Q_{r_n}(\bar x,t)}|a|^5dxds\Bigg)^\frac3{10}\\
&+C_2r_{n+1}^{6-\frac\delta2}\Bigg(\sup_{t-{r_{n+1}^2}<s<t}\int\limits_{r_n<|x-\bar x|<\frac14}\frac{|v(x,s)|^2}{|x|^4}dx\Bigg)^\frac32\\
&+C_2r_{n+1}^{4-\frac\delta2}\int\limits_{t-{r_{n+1}^2}}^t\Bigg(\int\limits_{r_n<|x-\bar x|<\frac14}\frac{|v(x,s)a(x,s)|}{|x|^4}dx\Bigg)^\frac32ds\\
&+2^{9}C_2r_{n+1}^{4-\frac\delta2}\int\limits_{Q_\frac14(\bar x,t)}|v|^3+|q|^\frac32dxds\\
&+2^{\frac{39}{5}}C_2r_{n+1}^\frac{44-5\delta}{10}\Bigg(\int\limits_{Q_\frac14(\bar x,t)}|v|^3dxds\Bigg)^\frac12\Bigg(\int\limits_{Q_\frac14(\bar x,t)}|a|^5dxds\Bigg)^\frac3{10}\\
=\ &J_1+\ldots\ J_6.
\end{align*}
We now estimate the right hand side term by term. We have 
\begin{equation*}
J_1\leq 2^{5-\delta}C_2C_3C_B^\frac32\ep_*r_{n+1}^\frac{9-\delta}2=D_5(\delta)\ep_*^\frac13\ep_*^\frac23r_{n+1}^3\leq \frac1{12}\ep_*^\frac23r_{n+1}^3.
\end{equation*}
For $J_2$, we have
\begin{equation*}
J_2\leq 2^{\frac{5-\delta}{2}}C_2C_3^\frac12C_B^\frac34\ep_*^\frac45r_{n+1}^{3-\delta}=D_6(\delta)\ep_*^\frac2{15}\ep_*^\frac23r_{n+1}^{3-\delta}\leq \frac1{12}\ep_*^\frac23r_{n+1}^{3-\delta}.
\end{equation*}
For $J_3$, we decompose into rings: we have
\begin{align*}
\sup_{t-{r_{n+1}^2}<s<t}\int\limits_{r_n<|x-\bar x|<\frac14}\frac{|v(x,s)|^2}{|x|^4}dx
\leq\ &\sum_{k=2}^{n-1}\sup_{t-{r_{k}^2}<s<t}\int\limits_{r_{k+1}<|x-\bar x|<r_k}\frac{|v(x,s)|^2}{|x|^4}dx\\
\leq\ &\sum_{k=2}^{n-1}r_{k+1}^{-4}C_B\ep_*^\frac23r_k^{3-\frac23\delta}\\
\leq\ &2^{3-\frac23\delta}C_B\ep_*^\frac23\sum_{k=2}^{n-1}r_{k+1}^{-1-\frac23\delta}\\
\leq\ &\frac{2^{3-\frac23\delta}}{2^{1+\frac23\delta}-1}C_B\ep_*^\frac23r_{n+1}^{-1-\frac23\delta},
\end{align*}
so that
\begin{equation*}
J_3\leq C_2\left(\frac{2^{3-\frac23\delta}}{2^{1+\frac23\delta}-1}\right)^\frac32C_B^\frac32\ep_*r_{n+1}^{\frac{9-3\delta}{2}}\leq D_7(\delta)\ep_*^\frac13\ep_*^\frac23r_{n+1}^{3-\delta}\leq \frac1{12}\ep_*^\frac23r_{n+1}^{3-\delta}.
\end{equation*}
For $J_4$, we have,
\begin{align*}
&\int\limits_{t-{r_{n+1}^2}}^t\Bigg(\int\limits_{r_n<|x-\bar x|<\frac14}\frac{|v(x,s)a(x,s)|}{|x|^4}dx\Bigg)^\frac32ds\\
\leq\ &C\int\limits_{t-{r_{n+1}^2}}^t\Bigg(\sum_{k=2}^{n-1}\int\limits_{r_{k+1}<|x-\bar x|<r_k}\frac{|v(x,s)a(x,s)|}{|x|^4}dx\Bigg)^\frac32ds\\
\leq\ &C\int\limits_{t-{r_{n+1}^2}}^t\Bigg(\sum_{k=2}^{n-1}r_{k+1}^{-4}\int\limits_{r_{k+1}<|x-\bar x|<r_k}|v(x,s)a(x,s)|dx\Bigg)^\frac32ds\\
\leq\ &C\int\limits_{t-{r_{n+1}^2}}^t\!\!\left(\sum_{k=2}^{n-1}r_{k+1}^{-4+\frac9{10}}\Bigg(\int\limits_{r_{k+1}<|x-\bar x|<r_k}|v(x,s)|^2dx\Bigg)^\frac12\!\!\!\Bigg(\int\limits_{
|x-\bar x|<r_k}|a(x,s)|^5dx\Bigg)^\frac1{5}\right)^\frac32\\
\leq\ &C\sup_{t-r_{n+1}^2<s<t}\left(\sum_{k=2}^{n-1}r_{k+1}^{-4+\frac9{10}}\big(\ep_*^\frac23r_{k}^{3-\frac23\delta}\big)^\frac12\right)^\frac32
\int\limits_{t-{r_{n+1}^2}}^t\Bigg(\int\limits_{B_1(0)}|a(x,s)|^5dx\Bigg)^\frac3{10}ds\\
\leq\ &C\left(\sum_{k=2}^{n-1}r_{k+1}^{-4+\frac9{10}}\big(\ep_*^\frac23r_{k}^{3-\frac23\delta}\big)^\frac12\right)^\frac32\ep_*^\frac32r_{n+1}^\frac75.
\end{align*}
This yields, 
\begin{align*}
J_4\leq\ &CC_2r_{n+1}^{4-\frac\delta2}\ep_*^2\left(\sum_{k=2}^{n-1}r_{k+1}^{-4+\frac9{10}}r_{k}^{\frac32-\frac\delta3}\right)^\frac32r_{n+1}^\frac75\\
\leq\ &2^{\frac94-\frac\delta2}CC_2\ep_*^2r_{n+1}^{4+\frac75-\frac\delta2}\left(\sum_{k=2}^{n-1}r_{k+1}^{-\frac85-\frac\delta3}\right)^\frac32\\
\leq\ &2^{\frac94-\frac\delta2}(2^{\frac85+\frac\delta3}-1)^{-\frac32}CC_2\ep_*^2r_{n+1}^{3-\delta}=D_8(\delta)\ep_*^\frac43\ep_*^\frac23r_{n+1}^{3-\delta}
\leq\frac1{12}\ep_*^\frac23r_{n+1}^{3-\delta}
\end{align*} 
The control of $J_5$ and $J_6$ is straightforward as the quantities for $v$ and $q$ are on the large scale cylinder $Q_{\frac14}(\bar x,t)$. This implies \eqref{e.ak} for $k=n+1$ and hence concludes the proof of Theorem \ref{theo.morrey}.
\end{proof}

We see in this argument, see in particular the control of $J_2$ and $J_4$ above, that the assumption $a\in L^5(Q_1)$ is critical to get the propagation of the Morrey-type bound.

\section{Local space-time boundedness near initial time}
\label{sec.boundedness}

Throughout this section, we define
$$L(f)(x,t):= \int\limits_{0}^{t} e^{(t-s)\Delta} f(\cdot,s) ds$$
and
 $$L(\nabla\cdot F):=\int\limits_{0}^{t} \partial_{j}e^{(t-s)\Delta} F_{j}(\cdot,s) ds.  $$
 Here,
 \begin{equation}\label{wholespaceheatkernel}
 e^{t\Delta}:= \frac{1}{(4\pi t)^{\frac{3}{2}}} \exp\Big(-\frac{|x|^2}{4t}\Big)\ast.
 \end{equation}
 Moreover, $(\mathcal R_i)_{i\in\{1,\ldots\, 3\}}$ denote the Riesz transforms.
\subsection{Morrey space preliminaries}
\begin{proposition}\label{Morreyextension}
Suppose $f\in L^{1}(Q_{1})$ and that there exists $0<\delta<5$ such that
\begin{equation}\label{Morreysmallscales}
\|f\|_{\delta, \bar{Q}_{\frac{1}{2}}(0,0)}:=\sup_{\stackrel{0<r\leq\frac{1}{4},}{(\bar{x}, t)\in \bar{Q}_{\frac{1}{2}}(0,0)}} r^{\delta-5}\int\limits_{Q_{r}(\bar{x}, t)} |f|dxdt<\infty. 
\end{equation}
Then $\tilde{f}= \chi_{Q_{\frac{1}{2}}(0,0)} f$ is such that
\begin{equation}\label{extendedfunctionMorrey}
\sup_{0<r,\, (\bar{x}, t)\in \mathbb{R}^4} r^{\delta-5}\int\limits_{Q_{r}(\bar{x}, t)} |\tilde{f}|dxdt\leq C(\delta, \|f\|_{L^{1}(Q_{1})},\|f\|_{\delta, \bar{Q}_{\frac{1}{2}}(0,0)}). 
\end{equation}
Furthermore,
\begin{equation}\label{extendedMorreynotbackward}
\sup_{0<r,\, (\bar{x}, t)\in \mathbb{R}^4} r^{\delta-5}\int\limits_{t-r^2}^{t+r^2}\int\limits_{B_{r}(\bar{x}, t)} |\tilde{f}|dxdt\leq C(\delta, \|f\|_{L^{1}(Q_{1})},\|f\|_{\delta, \bar{Q}_{\frac{1}{2}}(0,0)}).
\end{equation}

\end{proposition}
\begin{proposition}\label{Morreygradheatbounded}
Suppose \begin{equation}\label{supportg}
\supp g\subset B_{\frac{1}{2}}(0)\times (-\tfrac{1}{4}, \tfrac{1}{4}),
\end{equation}
\begin{equation}\label{gintegrable}
\int\limits_{-\frac{1}{4}}^{\frac{1}{4}}\int\limits_{B_{\frac{1}{2}}(0)} |g| dyds<\infty
\end{equation}
 and that there exists $\delta\in (0,1)$ such that 
\begin{equation}\label{gMorrey}
\|g\|_{\delta}:=\sup_{0<r,\, ({\bar x}, t)\in \mathbb{R}^4} r^{\delta-5}\int\limits_{t-r^2}^{t+r^2}\int\limits_{B_{r}(\bar x)} |g|dxdt\, < \infty.
\end{equation}
Then, under the above assumptions we have
\begin{equation}\label{gheatkernelderivativebounded}
\sup_{(x,t)\in\mathbb{R}^4} \int\limits_{-\infty}^{\infty} \int\limits_{\mathbb{R}^3} \frac{|g(y,s)|}{(|x-y|^2+ |t-s|)^2} dyds\leq \max{\Big(C(\delta)\|g\|_{\delta},16\int\limits_{-\frac{1}{4}}^{\frac{1}{4}}\int\limits_{B_{\frac{1}{2}}(0)} |g| dyds \Big)}.
\end{equation}
\begin{proof}

\noindent{\bf \underline{Case 1:} $(x,t)\in B_{1}(0) \times (-1,1)$.}\\
The proof of this case is along the lines of \cite{OLeary} (see also \cite{Kukavicamorrey}-\cite{Kukavicamorreyforce}). The difference is that we exploit the compact support of $g$ to control the integral at large distances. Clearly, $B_{\frac{1}{2}}(0) \times(-\frac{1}{4}, \frac{1}{4})\subset B_{2}(x)\times (t-4, t+4)$.
Thus, \begin{equation}\label{splitintoannuli}
\int\limits_{-\infty}^{\infty} \int\limits_{\mathbb{R}^3} \frac{|g(y,s)|}{(|x-y|^2+ |t-s|)^2} dyds= \sum_{k=0}^{\infty} \int\limits_{A_{k}} \frac{|g(y,s)|}{(|x-y|^2+ |t-s|)^2} dyds
\end{equation}
with
$$
A_{k}:=\{(y,s): 8^{-k}< |x-y|^2+|t-s|< 8^{1-k}\}.
$$
With this and (\ref{gMorrey}) we have
\begin{align}
\label{gannuliest}
\int\limits_{A_{k}} \frac{|g(y,s)|}{(|x-y|^2+ |t-s|)^2} dyds\leq\ & 8^{2k} \int\limits_{t-8^{1-k}}^{t+8^{1-k}} \int\limits_{B_{\sqrt{8^{1-k}}}(x)} |g(y,s)| dyds\nonumber\\
\leq\ & 8^{\frac{(\delta-1)k}{2}+\frac{5-\delta}{2}}\|g\|_{\delta}.
\end{align}
From (\ref{splitintoannuli}) and (\ref{gannuliest}) (and using the fact that $\delta\in (0,1)$), we conclude
$$
\sup_{(x,t)\in B_{1}(0)\times (-1,1)} \int\limits_{-\infty}^{\infty} \int\limits_{\mathbb{R}^3} \frac{|g(y,s)|}{(|x-y|^2+ |t-s|)^2} dyds\leq \frac{8^{\frac{5-\delta}{2}}\|g\|_{\delta}}{1- 8^{\frac{\delta-1}{2}}}.$$

\noindent{\bf \underline{Case 2:} $(x,t)\in \mathbb{R}^4 \setminus(B_{1}(0)\times (-1,1))$.}\\
In this case, $\inf_{(y,s)\in B_{\frac{1}{2}}(0)\times (-\frac{1}{4}, \frac{1}{4})} (|x-y|^2+|t-s|)\geq \frac{1}{4}$, which implies
$$ \int\limits_{-\infty}^{\infty} \int\limits_{\mathbb{R}^3} \frac{|g(y,s)|}{(|x-y|^2+ |t-s|)^2} dyds\leq 16 \int\limits_{-\frac{1}{4}}^{\frac{1}{4}}\int\limits_{B_{\frac{1}{2}}(0)} |g| dyds.$$
This concludes the proof.
\end{proof}

\end{proposition}
\begin{corollary}\label{MorreyboundsOseen}
Let $\varphi \in C_{0}^{\infty}(B_{\frac{1}{2}}(0))$ with $0\leq\varphi\leq 1$ and $\varphi=1$ on $B_{\frac{1}{3}}(0)$. Suppose that $v\in L^{3}(B_1(0)\times (-1,1))$ and there exists $S^*\in (0,\frac{1}{4})$ and $\delta\in (0,\frac{3}{2})$ such that:
\begin{equation}\label{MorreyvL3}
\||v|^3\|_{\delta,\bar{Q}_{\frac{1}{2}}(0,S^*)}:=\sup_{\stackrel{0<r\leq\frac{1}{4},}{(\bar{x}, t)\in \bar{Q}_{\frac{1}{2}}(0,S^*)}} r^{\delta-5}\int\limits_{Q_{r}(\bar{x}, t)} |v|^3dxdt<\infty.
\end{equation} 
Then,
\begin{equation}\label{heatdivMorreybounded}
\|L(\nabla\cdot(\varphi v\otimes v))\|_{L^{\infty}(\mathbb{R}^3\times (0,S^*))}\leq C(\delta, \||v|^3\|_{\delta,\bar{Q}_{\frac{1}{2}}(0,S^*)})
\end{equation}
\begin{equation}\label{OseenMorreybounded}
\|L(\nabla\mathcal{R}_{i}\mathcal{R}_{j}(\varphi v\otimes v))\|_{L^{\infty}(\mathbb{R}^3\times (0,S^*))}\leq C(\delta, \||v|^3\|_{\delta,\bar{Q}_{\frac{1}{2}}(0,S^*)}).
\end{equation}
\end{corollary}
\begin{proof}
This result follows from the previous two Propositions and the known fact that in $\mathbb{R}^3$ $$\partial_{j}e^{\Delta t}\,\,\,\,\textrm{and}\,\,\,\, e^{\Delta t} \mathbb{P}\nabla \cdot$$ are represented by convolution operators with kernels bounded by $\frac{1}{(|x|^2+t)^2}.$ See \cite{LR02}, for example.
\end{proof}
\subsection{Linear bootstrap arguments}

 In this subsection, we prove the following result.
 \begin{theorem}\label{localboundednesscriticalStokes}
 Suppose  $a$ is a divergence free vector field satisfying $a\in L^{5}(\mathbb{R}^3 \times (0,\infty))$ and $\sup_{0<s} s^{\frac{1}{5}}\|a(\cdot,s)\|_{L^{5}(\mathbb{R}^3)}<\infty$.\\
 Let $v \in C_{w}([-1,1]; L^{2}(B_{2}(0)))$, $\nabla v \in L^{2}(B_{2}(0)\times [-1,1]) $ and $q\in L^{1}(-1,1; L^{1}(B_{2}(0)))$ satisfy (in the distributional sense):
 \begin{equation}\label{stokescriticaldrift}
 \partial_{t} v-{\Delta} v+ a\cdot\nabla v+v\cdot\nabla a+{v\cdot\nabla v}+\nabla q=0,
 \end{equation}
 \begin{equation}\label{divfreezeroinitialdata}
 \nabla\cdot v=0\,\,\,\,\textrm{in}\,\, B_{2}(0)\times (0,1)\,\,\,\,\textrm{and}\,\,\,\,\, v(\cdot,0)=0\,\,\,\,\,\textrm{in}\,\, B_{2}(0).
 \end{equation}
 Furthermore, assume  there exists $\delta\in (0,\frac{3}{2})$ and $S^*\in (0,\frac{1}{4})$ such that \begin{equation}\label{MorreyuL3}
\||v|^3\|_{\delta,\bar{Q}_{\frac{1}{2}}(0,S^*)}:=\sup_{\stackrel{0<r\leq\frac{1}{4},}{(\bar{x}, t)\in \bar{Q}_{\frac{1}{2}}(0,S^*)}} r^{\delta-5}\int\limits_{Q_{r}(\bar{x}, t)} |v|^3dxdt<\infty.
\end{equation} 
 Under the above hypothesis, there is a universal constant $\ep_{**}\in (0,\infty)$ such that if
 \begin{equation}\label{smalldrift}
 \|a\|_{L^{5}(\mathbb{R}^3\times(0,\infty))}+ \sup_{s>0}s^{\frac{1}{5}}\|a(\cdot,s)\|_{L^{5}(\mathbb{R}^3)}\leq\ep_{**}
 \end{equation}
 then $v \in L^{\infty}(B_{\frac{1}{3}}(0) \times (0,S^*)).$
 Furthermore,
\begin{align}\label{uboundedest}
\begin{split}
& \|v\|_{L^{\infty}(B_{\frac{1}{3}}(0)\times (0,S^*))}\\
 &\ \leq C\big(\|q\|_{L^{1}(B_{2}(0)\times (0,1))},\|v\|_{L^{\infty}(0,1; L^{2}(B_{2}(0)))},\|\nabla v\|_{L^{2}(B_{2}(0)\times (0,1)))},\||v|^3\|_{\delta,\bar{Q}_{\frac{1}{2}}(0,S^*)}\big).
\end{split}
\end{align}
If in addition $q\in L^\frac32(B_2(0)\times(-1,1))$, then there exists $\nu\in(0,\frac12)$ such that
\begin{align}\label{e.holdertheo}
\begin{split}
& [v]_{C^{0,\nu}_{par}(B_{\frac{1}{3}}(0)\times [0,S^*])}\\
&\ \leq C\big(\|q\|_{L^\frac32(B_2(0)\times(-1,1))},\|v\|_{L^{\infty}(0,1; L^{2}(B_{2}(0)))},\|\nabla v\|_{L^{2}(B_{2}(0)\times (0,1)))},\||v|^3\|_{\delta,\bar{Q}_{\frac{1}{2}}(0,S^*)}\big).
\end{split}
\end{align}
 \end{theorem}
We now state some known linear heat estimates involving the above operators. For a detailed proof, we refer the reader to  Appendix D of \cite{RRS16}.

\begin{proposition}\label{estimatesnondivergence}
Let $T\in (0,\infty)$. For $q=p$, we have
\begin{equation}\label{ctsintime}
\|L(f)(\cdot,t)\|_{L^p(\mathbb{R}^3)}\leq \int\limits_{0}^{t} \|f(\cdot, s)\|_{L^{p}(\mathbb{R}^3)} ds.
\end{equation}
For $\frac{3}{q}+\frac{2}{s}\geq  \frac{3}{p}+\frac{2}{r}-2$ ($p<q<\infty$), we have
\begin{equation}\label{nondivsobolevembedding}
{\|L(f)\|_{L^{s}((0,T); L^{q}(\mathbb{R}^3))}\leq C(r,p,s,q)\|f\|_{L^{r}((0,T); L^{p}(\mathbb{R}^3))}.}
\end{equation}
For $\frac{3}{p}+\frac{2}{r}<2$, one has
\begin{equation}\label{boundednondiv}
\|L(f)\|_{L^{\infty}(\mathbb{R}^3 \times (0,T))}\leq C(r,p)\|f\|_{L^{r}((0,T); L^{p}(\mathbb{R}^3))}.
\end{equation}
One also has
\begin{equation}\label{boundedendpoint}
{\|L(f)\|_{L^{\infty}(\mathbb{R}^3 \times (0,T))}\leq C\|f\|_{L^{1}((0,T); L^{\infty}(\mathbb{R}^3))}}
\end{equation}
\end{proposition}
\begin{proposition}\label{estimatesdivergence}
For $\frac{1}{q}= \frac{1}{p}-\frac{1}{5}$ ($p<q<\infty$), we have
\begin{equation}\label{divsobolevembedding}
{\|L(\nabla\cdot F)\|_{L^{q}(\mathbb{R}^3 \times (0,T))}\leq C(p,q)\|F\|_{L^{p}(\mathbb{R}^3 \times (0,T))}}.
\end{equation}
For $p\in(5,\infty)$, one has
\begin{equation}\label{boundeddiv}
{\|L(\nabla\cdot F)\|_{L^{\infty}(\mathbb{R}^3 \times (0,T))}\leq C(p)\|F\|_{L^{p}(\mathbb{R}^3 \times (0,T))}}.
\end{equation}
\end{proposition}
\begin{proposition}\label{tensorproductheatest}
For  $q\in[\frac{4}{5},\infty)$, we have
\begin{equation}\label{criticalestfinite}
\|L(\nabla\cdot(a\otimes b))\|_{L^{q}(\mathbb{R}^3 \times (0,T))} \leq C(q)\|a\|_{L^{5}(\mathbb{R}^3 \times (0,T))} \|b\|_{L^{q}(\mathbb{R}^3 \times (0,T))},
\end{equation}
\begin{equation}\label{criticalestbounded}
\|L(\nabla\cdot(a\otimes b))\|_{L^{\infty}(\mathbb{R}^3 \times (0,T))} \leq C_{univ}\big(\sup_{0<s<\infty}s^{\frac{1}{5}}\|a(\cdot,s)\|_{L^{5}(\mathbb{R}^3)}\big)\|b\|_{L^{\infty}(\mathbb{R}^3 \times (0,T))}.
\end{equation}
Here, $C_{univ}\in(0,\infty)$ is a universal constant.
\end{proposition}

\begin{proposition}[H\"older estimates]\label{prop.holderbootstrap}
For all $r,\, p\in[1,\infty]$ such that $\frac2r+\frac3p<2$, we have
\begin{equation}\label{e.holderone}
[L(f)]_{C^{0,\nu}_{par}(\R^3\times [0,T])}\leq C(p,T)\|f\|_{L^r(0,T;L^p(\R^3))},
\end{equation}
for all $\nu\in(0,\min(1-\frac1r-\frac3{2p},\frac12))$. 
Moreover, for all $r,\, p\in[1,\infty]$ such that $\frac2r+\frac3p<1$, we have
\begin{equation}\label{e.holdertwo}
[L(\nabla\cdot F)]_{C^{0,\nu}_{par}(\R^3\times [0,T])}\leq C(p,T)\|F\|_{L^r(0,T;L^p(\R^3))},
\end{equation}
for all $\nu\in(0,\frac12-\frac1r-\frac3{2p})$. 
\end{proposition}

\begin{proof}[Proof of Proposition \ref{prop.holderbootstrap}]
Our first two claims are that for all $u_0\in L^p(\R^3)$, $p\in[1,\infty]$, for all $t\in(0,\infty)$,
\begin{align}
\label{e.claimholder}
[e^{t\Delta}u_0]_{C^{0,2\nu}(\R^3)}+\sup_{h\in(0,T),\ x\in\R^3}\frac{|e^{(t+h)\Delta}u_0-e^{t\Delta}u_0|}{|h|^{\nu}}\leq\ &\frac{C(p,\nu,T)}{t^{\frac3{2p}+\nu}}\|u_0\|_{L^p},\\
[\nabla e^{t\Delta}u_0]_{C^{0,2\nu}(\R^3)}+\sup_{h\in(0,T),\ x\in\R^3}\frac{|\nabla e^{(t+h)\Delta}u_0-\nabla e^{t\Delta}u_0|}{|h|^{\nu}}\leq\ &\frac{C(p,\nu,T)}{t^{\frac3{2p}+\nu+\frac12}}\|u_0\|_{L^p}.\label{e.claimholderbis}
\end{align}
These two estimates are simple consequences of estimates for the heat semigroup. 
We give the proof of \eqref{e.holdertwo}, which relies on \eqref{e.claimholderbis}. The proof of \eqref{e.holderone} is similar and relies on \eqref{e.claimholder} instead. Assume that $\frac2r+\frac3p<1$ and $\nu\in(0,1-\frac1r-\frac3{2p})$. Then, 
\begin{align*}
[L(\nabla\cdot F)]_{C^{0,\nu}_{par}(\R^3\times [0,T])}\leq\ &\sup_{t\in(0,T)}[L(\nabla\cdot F)(\cdot,t)]_{C^{0,2\nu}(\R^3)}\\
&+\sup_{t\in (0,T)}\sup_{h\in(0,T),\ x\in\R^3}\frac{|L(\nabla\cdot F)(x,t+h)-L(\nabla\cdot F)(x,t)|}{|h|^{\nu}}.
\end{align*}
On the one hand, we have for all $t\in (0,T)$,
\begin{align*}
[L(\nabla\cdot F)(\cdot,t)]_{C^{0,2\nu}(\R^3)}\leq\ &\int\limits_0^t[e^{(t-s)\Delta}\nabla\cdot F]_{C^{0,2\nu}(\R^3)}ds\\
\leq\ &C(p,\nu)\int\limits_0^t(t-s)^{-\frac3{2p}-\nu-\frac12}\|F(\cdot,s)\|_{L^p(\R^3)}ds\\
\leq\ &C(p,\nu)\Bigg(\int\limits_0^t(t-s)^{-(\frac3{2p}+\nu+\frac12)\frac r{r-1}}ds\Bigg)^{\frac{r-1}r}\|F\|_{L^r(0,T;L^p(\R^3))},
\end{align*}
which belongs to $L^\infty(0,T)$. On the other hand, for all $h\in (0,T)$, for all $t\in (0,T)$,
\begin{align*}
&\left\|\frac{L(\nabla\cdot F)(x,t+h)-L(\nabla\cdot F)(x,t)}{|h|^{\nu}}\right\|_{L^\infty(\R^3)}\\
\leq\ &\int\limits_0^t\left\|\frac{e^{(t+h-s)\Delta}\nabla\cdot F-e^{(t-s)\Delta}\nabla\cdot F}{|h|^\nu}\right\|_{L^\infty(\R^3)}ds+\int\limits_t^{t+h}\left\|\frac{e^{(t+h-s)\Delta}\nabla\cdot F}{|h|^\nu}\right\|_{L^\infty(\R^3)}ds\\
=\ &I_1(h,t)+I_2(h,t).
\end{align*}
For $I_1$, we immediately have that 
\begin{equation*}
\sup_{h\in (0,T)} I_1(h,t)\leq C(p,\nu,T)\Bigg(\int\limits_0^t(t-s)^{-(\frac3{2p}+\nu+\frac12)\frac r{r-1}}ds\Bigg)^{\frac{r-1}r}\|F\|_{L^r(0,T;L^p(\R^3))},
\end{equation*}
which is bounded in $t\in (0,T)$. The term $I_2$ is a remainder term. 
A direct computation leads to 
\begin{equation*}
I_2(h,t)\leq C(p,\nu)|h|^{-\nu-\frac{3}{2p}-\frac12+\frac{r-1}{r}},
\end{equation*}
which is bounded uniformly in $t\in (0,T)$, $h\in (0,T)$. 
\end{proof}

In order to prove Theorem  \ref{localboundednesscriticalStokes} we will first localise $v$ in space with a cut-off function $\varphi$ (in particular we will consider $\tilde{v}=\varphi v$).
When considering the pressure, we will encounter objects such as $$\nabla(\varphi\mathcal{R}_{i}\mathcal{R}_{j}((v_{i}a_{j}+a_{i}v_{j}+ v_{i}v_{j})\chi_{B_2(0)})).$$  Using Propositions \ref{estimatesnondivergence}-\ref{estimatesdivergence} we see that if $$\partial_{t}w-\Delta w=\nabla (\varphi\mathcal{R}_{i}\mathcal{R}_{j}((v_{i}a_{j}+a_{i}v_{j})\chi_{B_2(0)}))\,\,\,\,\,\textrm{in}\,\,\,\mathbb{R}^3\times(0,1)\,\,\,\,\,\textrm{and}\,\,\,\,\,w(\cdot,0)=0$$ then the criticality of $a$ means that $w$ has the same integrability as $v$, which is troublesome with regards to improving the integrability of $\tilde{v}$. 
To avoid this issue, it will be advantageous to split $\varphi\mathcal{R}_{i}\mathcal{R}_{j}((v_{i}a_{j}+a_{i}v_{j}+ v_{i}v_{j})\chi_{B_2(0)})$ with one piece being ``well localised''
$$\mathcal{R}_{i}\mathcal{R}_{j}(\tilde{v}_{i}a_{j}+a_{i}\tilde{v}_{j}+ \varphi v_{i}v_{j}).$$ Furthermore, another key advantage\footnote{This particular advantage was also exploited by Kukavica in \cite{Kukavicamorreyforce}.} is that we can apply \eqref{OseenMorreybounded} to $\nabla \mathcal{R}_{i}\mathcal{R}_{j}(\varphi v_{i}v_{j})$.   This decomposition allows us to consider the invertible operator $$\tilde{v}-L(\nabla\cdot(\tilde{v}\otimes a+ a\otimes \tilde{v}))+ \int\limits_{0}^{t} \nabla e^{(t-s)\Delta} \mathcal{R}_{i}\mathcal{R}_{j}(\tilde{v}_{i}a_{j}+\tilde{v}_{j}a_{i})ds=F,$$
where $F$ depends on $v$ and improves its integrability at each stage of the bootstrap. Related bootstrap arguments were used by Seregin in \cite{Ser06}. 
The splitting of the pressure from Lemma \ref{pressure localisation} in Appendix \ref{app.a} gives us what we need. 

Now, we state a Lemma regarding invertibility of a certain linear operator, that will play a key role in bootstrapping the integrability of $v$. The proof essentially follows from Proposition \ref{tensorproductheatest}.
\begin{lemma}\label{operatorinvertiblity}
Let $T\in (0,\infty)$. 
Define $$L_{a}(u):= L(\nabla\cdot(u\otimes a+ a\otimes u))+ \int\limits_{0}^{t} \nabla e^{(t-s)\Delta} \mathcal{R}_{i}\mathcal{R}_{j}(u_{i}a_{j}+u_{j}a_{i})ds.$$
Then for every $\frac{5}{4}\leq q\leq \infty$ there exists $\ep(q)\in(0,\infty)$ such that if 
\begin{equation}\label{smallnessforinvertibility}
{\|a\|_{L^{5}(\mathbb{R}^3 \times (0,T))}+\sup_{0<s<1} s^{\frac{1}{5}} \|a(\cdot,s)\|_{L^5(\mathbb{R}^3)}\leq\varepsilon(q)}
\end{equation}
then $I-L_{a}: L^{q}(\mathbb{R}^3 \times (0,T)) \rightarrow L^{q}(\mathbb{R}^3 \times (0,T))$
is invertible.
Moreover, if
\begin{equation}\label{smallnessforinvertibilityintersect}
\|a\|_{L^{5}(\mathbb{R}^3 \times (0,T))}+\sup_{0<s<T} s^{\frac{1}{5}} \|a(\cdot,s)\|_{L^5(\mathbb{R}^3)}\leq\min{(\varepsilon(p), \varepsilon(q))}
\end{equation}
then \begin{equation}\label{equalonintersection}
{\textrm{on}\,\,\,\,L^{p}(\mathbb{R}^3 \times (0,T) )\cap L^{q}(\mathbb{R}^3 \times (0,T)),\,\,\,\,\, (I-L_{a})^{-1}_{L^{p}\rightarrow L^{p}}= (I-L_{a})^{-1}_{L^{q}\rightarrow L^{q}}}.
\end{equation}
\end{lemma}
\begin{remark}
Let us comment on showing \eqref{equalonintersection}.
Suppose $f\in L^{p}(\mathbb{R}^3 \times (0,T))\cap L^{q}(\mathbb{R}^3 \times (0,T))$ with $p,q\in [\frac{5}{4}, \infty]$.
Define the Picard iterates
\begin{equation}\label{zeroPicard}
P_{0}=f
\end{equation}
\begin{equation}\label{kthPicard}
P_{k+1}=f+L_{a}(P_{k}).
\end{equation}
Under the smallness assumption on $a$, we have 
\begin{equation}\label{picardLpconverg}
\lim_{k\rightarrow\infty}\|P_{k}- (I-L_{a})^{-1}_{L^{p}\rightarrow L^{p}}(f)\|_{L^{p}(\mathbb{R}^3\times (0,T))}=0
\end{equation}
and
\begin{equation}\label{picardLqconverg}
\lim_{k\rightarrow\infty}\|P_{k}- (I-L_{a})^{-1}_{L^{q}\rightarrow L^{q}}(f)\|_{L^{q}(\mathbb{R}^3\times (0,T))}=0.
\end{equation}
Thus we obtain (\ref{equalonintersection}).
\end{remark}

We now turn to the proof of the main result of this section.
\begin{proof}[Proof of Theorem \ref{localboundednesscriticalStokes}]
\textbf{\underline{Step 1:} Spatial localisation.}\\
Let $\varphi\in C_{0}^{\infty}(B_{\frac{1}{2}}(0))$ be such that $0\leq \varphi\leq 1$ and $\varphi=1$ on $B_{\frac{1}{3}}(0)$. Define $\tilde{v}:= \varphi v$. Then we have that $\tilde{v}$ satisfies the following equation in the sense of distributions in $\mathbb{R}^3\times (0,1)$:
\begin{align*}\label{localinspace}
\begin{split}
\partial_{t}\tilde{v}-\Delta{\tilde{v}}+\nabla\cdot(a\otimes \tilde{v}+\tilde{v}\otimes a)= -&\nabla\cdot(\varphi v\otimes v)+(\nabla\varphi\cdot v)v+(\nabla\varphi\cdot v)a+(a\cdot\nabla\varphi)v\\-&\nabla(q\varphi)+q\nabla\varphi- {2\nabla\cdot(v\otimes\nabla \varphi)}+v\Delta\varphi.
\end{split}
\end{align*}
\begin{equation*}\label{tildeuinitial}
\tilde{v}(\cdot,0)=0.
\end{equation*}

\noindent\textbf{\underline{Step 2:} Decomposing the pressure.}\\
First we use a classical decomposition of the pressure namely
\begin{equation*}\label{pressuredecomp1}
q= \tilde{p}+ h(\cdot,t)+\hat{p}.
\end{equation*}
Here, 
\begin{equation*}\label{tildepdefine}
\tilde{p}:=\mathcal{R}_{i}\mathcal{R}_{j}((a_{i}v_{j}+a_{j}v_{i}) \chi_{B_{2}(0)})
\end{equation*}
and
\begin{equation*}\label{hatpdefine}
{\hat{p}:=\mathcal{R}_{i}\mathcal{R}_{j}((v_{i}v_{j}) \chi_{B_{2}(0)})}.
\end{equation*}
Using that $a\in L^{5}(\mathbb{R}^3 \times (0,1))$ and $v\in L^{\frac{10}{3}}(B_{2}(0) \times (0,1))$ we have
\begin{equation}\label{tensor}
a_{i}v_{j}\in L^{2}(B_{2}(0)\times (0,1)).
\end{equation}
Thus,\begin{equation}\label{ptildespaces}
\tilde{p}\in L^{2}(\mathbb{R}^3 \times (0,1)).
\end{equation}
{Using that $v\in C_{w}([0,1]; L^{2}(B_{2}(0)))$ and $\nabla v\in L^{2}((0,1)\times B_{2}(0))$ we have}
\begin{equation}\label{phatspaces}
\hat{p}\in L^{r}((0,1); L^{\lambda}(B_{2}(0)))\,\,\,\,\textrm{with}\,\,\,\,\,\frac{2}{r}+\frac{3}{\lambda}=3\,\,\,\,\, \lambda\in (1,3].
\end{equation}
Next we see that $h(\cdot,t)$ is a harmonic function in $B_{\frac{3}{2}}(0)$. Using this, we obtain for $k\in\mathbb{N}\setminus\{0\}$ that
\begin{align}\label{harmonicest}
\begin{split}
\|\nabla^{k} h\|_{L^{1}(0,1; L^{\infty}(B_{1}(0)))}\leq\ & C(k)(\|v\|_{L^{\infty}(0,1; L^{2}(B_{2}(0)))}+\|\nabla v\|_{L^{2}(B_{2}(0)\times (0,1))}\\
&+\|q\|_{L^{1}(0,1; L^{1}(B_{2}(0)))}+\|a\|_{L^{5}(\mathbb{R}^3\times (0,1))}\|v\|_{L^{\infty}(0,1; L^{2}(B_{2}(0)))}).
\end{split}
\end{align}
Now, we apply Lemma \ref{pressure localisation} to $\tilde{p}$ and $\hat{p}$  to get
\begin{align*}\label{tildeplocalisation}
\begin{split}
\varphi\tilde{p}=& -N\ast(\Delta \varphi \tilde{p})-N\ast((a_{i}v_{j}+a_{j}v_{i})\partial_{i}\partial_{j}\varphi)-2\partial_{j}N\ast(\tilde{p}\partial_{j}\varphi)\\&-2\partial_{j}N\ast((a_{i}v_{j}+a_{j}v_{i})\partial_{i}\varphi) +\mathcal{R}_{i}\mathcal{R}_{j}(a_{i}\tilde{v}_{j}+a_{j}\tilde{v}_{i}),
\end{split}
\end{align*}
\begin{align*}
\varphi\hat{p}=& -N\ast(\Delta \varphi \hat{p})-N\ast((v_{i}v_{j})\partial_{i}\partial_{j}\varphi)-2\partial_{j}N\ast(\hat{p}\partial_{j}\varphi)\\&-2\partial_{j}N\ast((v_{i}v_{j})\partial_{i}\varphi) +\mathcal{R}_{i}\mathcal{R}_{j}(\varphi v_{i}{v}_{j})
\end{align*}
Thus in $\mathbb{R}^3\times (0,2)$ we have,
\begin{equation}\label{tildeuequation2}
\partial_{t}\tilde{v}-\Delta \tilde{v}+ \nabla\cdot(a\otimes \tilde{v}+\tilde{v}\otimes a)+\nabla\mathcal{R}_{i}\mathcal{R}_{j}(a_{i}\tilde{v}_{j}+a_{j}\tilde{v}_{i})=\sum_{i=1}^{6} F_{i}+ \nabla\cdot(F_{7})\,
\end{equation}
\begin{equation}\label{tildeuequation2IC}
\tilde{v}(\cdot,0)=0.
\end{equation}
Here,
\begin{align*}\label{F1def}
\begin{split}
F_{1}:=& (\nabla\varphi\cdot v)a+(a\cdot\nabla\varphi)v+\tilde{p}\nabla\varphi+\nabla N\ast(\Delta \varphi \tilde{p})+\nabla N\ast((a_{i}v_{j}+a_{j}v_{i})\partial_{i}\partial_{j}\varphi)\\
&+2\nabla\partial_{j}N\ast(\tilde{p}\partial_{j}\varphi)+2\nabla\partial_{j}N\ast((a_{i}v_{j}+a_{j}v_{i})\partial_{i}\varphi),
\end{split}
\end{align*}
\begin{equation*}\label{F2def}
F_{2}:= v\Delta \varphi,
\end{equation*}
\begin{equation*}\label{F3def}
F_{3}:= h{\nabla} \varphi-\nabla(\varphi h),
\end{equation*}
\begin{equation*}\label{F4def}
{F_{4}:=\nabla N\ast(\Delta\varphi \hat{p})+ \nabla N\ast(v_{i}v_{j}\partial_{i}\partial_{j}\varphi)},
\end{equation*}
\begin{equation*}\label{F5def}
{F_{5}:=(\nabla\varphi\cdot v)v+ 2\nabla(\partial_{j}N\ast((v_iv_j\partial_{i}\varphi))+2\nabla\partial_{j}N\ast(\hat{p}\partial_{j}\varphi)},
\end{equation*}
\begin{equation*}\label{F6def}
{F_{6}:=-\nabla\cdot(\varphi v\otimes v)-\nabla(\mathcal{R}_{i}\mathcal{R}_{j}(\varphi v_{i}v_{j}))}
\end{equation*}
\begin{equation*}\label{F7def}
F_{7}:=-2v\otimes \nabla\varphi.
\end{equation*}

\noindent\textbf{\underline{Step 3:} First linear bootstrap.}\\
Using (\ref{tensor}) and (\ref{ptildespaces}) (along with the Hardy-Littlewood-Sobolev inequality and {Calder\'{o}n-Zygmund estimates}) we see that
\begin{equation*}\label{F1spaces}
F_{1}\in L^{2}(\mathbb{R}^3 \times (0,1)).
\end{equation*}
{Here, we have used that $\varphi$ has compact support, which implies that $\Delta\varphi \tilde{p}$ and $(a_{i}v_{j}+v_{i}a_{j})\partial_{i}\partial_{j}\varphi$ belong to $L^{2}(0,1; L^{\frac{6}{5}}(\mathbb{R}^3)).$} 
Thus, using Proposition \ref{estimatesnondivergence} we see that
\begin{equation*}\label{LF1spaces}
L(F_{1})\in L^{2}(\mathbb{R}^3 \times (0,1))\cap L^{10}(\mathbb{R}^3 \times (0,1)).
\end{equation*}
Using that $u$ is in the energy class, we see that
\begin{equation*}\label{F2F4spaces}
F_{2}, F_{7} \in L^{2}(\mathbb{R}^3 \times (0,1))\cap L^{\frac{10}{3}}(\mathbb{R}^3\times (0,1)).
\end{equation*}
Then we apply Propositions \ref{estimatesnondivergence}-\ref{estimatesdivergence} {and energy estimates} to see 
\begin{equation*}\label{LF2F4spaces}
L(F_{2}), L(\nabla\cdot F_{7}) \in L^{2}(\mathbb{R}^3 \times (0,1))\cap {L^{10}(\mathbb{R}^3\times (0,1))}.
\end{equation*}
{Using (\ref{harmonicest}) and the fact that $\varphi$ has compact support implies $$F_{3} \in L^{1}(0,1; L^{\infty}(\mathbb{R}^3))\cap L^{1}(0,1; L^{2}(\mathbb{R}^3)).$$  Using Proposition \ref{estimatesnondivergence} then gives}
\begin{equation*}\label{LF3spaces}
L(F_{3})\in L^{2}(\mathbb{R}^3 \times (0,1))\cap L^{\infty}(\mathbb{R}^3 \times (0,1)).
\end{equation*}
{Using (\ref{phatspaces}) and the Hardy-Littlewood-Sobolev inequality gives}
\begin{equation*}\label{F4spaces}
{F_{4}\in L^{r}(0,1; L^{\tilde{\lambda}}(\mathbb{R}^3))\,\,\,\,\,\textrm{with}\,\,\,\,\,\, \frac{2}{r}+\frac{3}{\tilde{\lambda}}=2\,\,\,\,\tilde{\lambda}\in \Big(\frac{3}{2}, \infty\Big).}
\end{equation*}
{Using this and Proposition \ref{estimatesnondivergence} we get}
\begin{equation*}\label{LF4spaces}
{L(F_{4})\in L^{10}(\mathbb{R}^3 \times (0,1))\cap L^{2}(\mathbb{R}^3 \times (0,1)).}
\end{equation*}
{Using (\ref{phatspaces}) and Calder\'{o}n-Zygmund estimates  gives}
\begin{equation*}\label{F5spaces}
{F_{5}\in L^{r}(0,1; L^{{\lambda}}(\mathbb{R}^3))\,\,\,\,\,\textrm{with}\,\,\,\,\,\, \frac{2}{r}+\frac{3}{{\lambda}}=3\,\,\,\,{\lambda}\in (1,3).}
\end{equation*}
{So Proposition \ref{estimatesnondivergence} and energy estimates give}
\begin{equation*}\label{LF5spaces}
{L(F_{5}) \in L^{5}(\mathbb{R}^3 \times (0,1))\cap L^{2}(\mathbb{R}^3\times (0,1)).}
\end{equation*}
{Next, using (\ref{MorreyuL3}), Corollary \ref{MorreyboundsOseen} and the fact that $\varphi$ has compact support, we obtain:}
\begin{equation*}\label{LF6spaces}
{L(F_{6})\in L^{\infty}(\mathbb{R}^3 \times (0,S^*))\cap L^{2}(\mathbb{R}^3 \times (0,S^*)).}
\end{equation*}
So we have
\begin{equation}\label{firstbootstrapforcespaces}
{\sum_{i=1}^{6} L(F_{i})+ L(\nabla\cdot F_{7})\in L^{2}(\mathbb{R}^3 \times (0,S^*))\cap L^{5}(\mathbb{R}^3 \times (0,S^*)).}
\end{equation}
Since $\tilde{v}\in L^{2}(\mathbb{R}^3 \times (0,S^*))$ satisfies (\ref{tildeuequation2})-(\ref{tildeuequation2IC}), we have
\begin{equation*}\label{uoperatorfixedpoint}
\tilde{v}=(I-L_{a})^{-1}_{L^2 \rightarrow L^2}(\sum_{i=1}^{6} L(F_{i})+ L(\nabla\cdot F_{7}))
\end{equation*}
Using Lemma \ref{operatorinvertiblity} and (\ref{firstbootstrapforcespaces}) we see
\begin{equation*}\label{uoperatorfixedpointL10}
{\tilde{v}=(I-L_{a})^{-1}_{L^{5} \rightarrow L^{5}}(\sum_{i=1}^{6} L(F_{i})+ L(\nabla\cdot F_{7}))}
\end{equation*}
In particular, $\tilde{v}\in L^{5}(\mathbb{R}^3 \times (0,S^*))$, which represents a gain in integrability.

\noindent\textbf{\underline{Step 4:} Second linear bootstrap.}\\
Without loss of generality we now assume {$v \in C_{w}([0,S^*]; L^{2}(B_{2}(0))\cap L^{5}(B_{2}(0)\times (0,S^*))$, $\nabla v \in L^{2}(B_{2}(0)\times [0,S^*]). $}
In this case, one can show that 
\begin{equation*}\label{F145spaces2}
{F_{1}, F_{4}, F_{5}\in L^{2}(\mathbb{R}^3 \times (0,S^*))\cap L^{\frac{5}{2}}(\mathbb{R}^3 \times (0,S^*))}
\end{equation*}
and
\begin{equation*}\label{F27spaces2}
{F_{2},F_{7} \in L^{2}(\mathbb{R}^3 \times (0,S^*))\cap L^{5}(\mathbb{R}^3 \times (0,S^*)).}
\end{equation*}
Using Propositions \ref{estimatesnondivergence}-\ref{estimatesdivergence} we see that
\begin{equation*}\label{secondbootstrapforcespaces}
{\sum_{i=1}^{6} L(F_{i})+ L(\nabla\cdot F_{7})\in L^{2}(\mathbb{R}^3 \times (0,S^*))\cap L^{10}(\mathbb{R}^3 \times (0,S^*)).}
\end{equation*}
Verbatim reasoning to the first linear bootstrap yields that
$\tilde{v}\in L^{10}(\mathbb{R}^3\times (0,S^{*}))$.

\noindent\textbf{\underline{Step 5:} Third linear bootstrap.}\\
Without loss of generality we now assume $v \in C_{w}([0,S^*]; L^{2}(B_{2}(0))\cap L^{10}(B_{2}(0)\times (0,S^*))$, $\nabla v\in L^{2}(B_{2}(0)\times [0,S^*]). $
In this case, one can show that the forcing terms $F_{i}$  belong to the following spaces:
\begin{equation*}\label{F1spaces3}
F_{1}\in L^{2}(\mathbb{R}^3 \times (0,S^*))\cap L^{\frac{10}{3}}(\mathbb{R}^3 \times (0,S^*))
\end{equation*}
\begin{equation*}\label{F2F7spaces3}
{F_{2}, F_{7}\in L^{2}(\mathbb{R}^3 \times (0,S^*))\cap L^{10}(\mathbb{R}^3 \times (0,S^*))}
\end{equation*}
and
\begin{equation*}\label{F4F5spaces3}
{F_{4}, F_{5}\in L^{2}(\mathbb{R}^3 \times (0,S^*))\cap L^{5}(\mathbb{R}^3 \times (0,S^*)).}
\end{equation*}
Using Propositions \ref{estimatesnondivergence}-\ref{estimatesdivergence} we see that
\begin{equation*}\label{thirdbootstrapforcespaces}
{\sum_{i=1}^{6} L(F_{i})+ L(\nabla\cdot F_{7})\in L^{2}(\mathbb{R}^3 \times (0,S^*))\cap L^{\infty}(\mathbb{R}^3 \times (0,S^*)).}
\end{equation*}
Verbatim reasoning to the first linear bootstrap yields that
$\tilde{v}\in L^{\infty}(\mathbb{R}^3\times (0,S^*))$.  This concludes the proof of the first part of the theorem.

\noindent\textbf{\underline{Step 6:} Final linear bootstraps.}\\
It remains to see the H\"older continuity up to initial time \eqref{e.holdertheo}. For this it is enough to assume $q\in L^r(B_2(0)\times(-1,1))$ for $r>0$, $r>1$. To fix the ideas, we assume $q\in L^\frac32(B_2(0)\times(-1,1))$. Then $h\in L^\frac32(B_2(0)\times(-1,1))$, and since $h(\cdot,t)$ is harmonic, we also have $h,\ \nabla h\in L^\frac32((0,1),L^\infty(\R^3)))$. 
From the fact that $\tilde v\in L^2(\R^3\times(0,S^*))\cap L^\infty(\R^3\times(0,S^*))$ and applying the H\"older estimates of Proposition \ref{prop.holderbootstrap}, we get that there exists $\nu\in(0,\frac12)$ such that
\begin{equation*}
\sum_{i=1}^6L(F_i)+L(\nabla\cdot F_7)\in C^{0,\nu}_{par}(B_{\frac13}\times[0,S^*]).
\end{equation*}
Using the invertibility of $I-L_a$ on $L^\infty(\R^3\times(0,t))$, we then get that for all $t\in (0,S^*)$,
\begin{align*}
\|\tilde v\|_{L^\infty(\R^3\times(0,t))}=\ &\|(I-L_a)^{-1}\big(\sum_{i=1}^6L(F_i)+L(\nabla\cdot F_7)\big)\|_{L^\infty(\R^3\times(0,t))}\\
\leq\ & C\big(\sum_{i=1}^6\|L(F_i)\|_{L^\infty(\R^3\times (0,t))}+\|L(\nabla\cdot F_7)\|_{L^\infty(\R^3\times (0,t))}\\
\leq\ & Ct^\nu,
\end{align*}
where here the constant $C$ is not universal, and depends in particular on the quantities in the right hand side of \eqref{e.holdertheo}. Notice that thanks to the weak continuity in time of $\tilde v$, we have that the previous bound implies 
\begin{equation*}
\sup_{s\in (0,S^*)}s^{-\nu}\|\tilde v(\cdot,s)\|_{L^\infty}<\infty.
\end{equation*} 
Hence, $a\otimes\tilde v\in L^{r}_t(0,S^*;L^5(\R^3))$, with $r\in (5,\frac{5}{1-5\nu})$. Therefore, one can apply the estimate \eqref{e.holdertwo} of Proposition \ref{prop.holderbootstrap} with $r\in (5,\frac{5}{1-5\nu})$ and $p=5$. Hence we get that there exists $\nu''\in (0,\frac12)$ such that $\tilde v\in C^{0,\nu''}_{par}(B_{\frac13}\times[0,S^*])$.
\end{proof}

\section{Proof of the main results}
\label{sec.proof}

This section is devoted to the proof of the main theorems. We first prove the local in space regularity near initial time, i.e. Theorem \ref{theo.main}. Then, we prove the concentration result, i.e. Theorem \ref{theo.conc}.

\subsection{Proof of Theorem \ref{theo.main}}

In this section, we choose the parameters in Theorem \ref{theo.morrey} as follows: $\delta=1$ and $E=1$. Hence there exists a universal constant $\ep_*=\ep_*(1,1)$ as given in Theorem \ref{theo.morrey}. This constant $\ep_*$ is now fixed for the remainder of this section. We let $\ep=\min(\ep_*,\ep_{**})$, where $\ep_{**}$ is the universal constant in Theorem \ref{localboundednesscriticalStokes}. 
The constant $\gamma_{univ}$ appearing in Theorem \ref{theo.main} will be chosen below, depending only on $\ep$, see \eqref{e.estgamma}.

Let $M\in(0,\infty)$ be fixed for the rest of this section. Let $u$ be any local energy solution with initial data $u_0$ as in Theorem \ref{theo.main}. Notice that by the bound of Proposition \ref{prop.lews}, we have
\begin{equation}\label{e.aprioriu}
\sup_{s\in(0,S_{lews})}\sup_{\bar x\in\R^3}\int\limits_{B_1(\bar x)}\frac{|u(x,s)|^2}{2}\, dx+\sup_{\bar x\in\R^3}\int\limits_0^{S_{lews}}\int\limits_{B_1(\bar x)}|\nabla u(x,s)|^2\, dx\, ds\leq K_1M^2
\end{equation}
with  $S_{lews}:=c_1^2\min(M^{-4},1)$. Furthermore, for the pressure associated to $u$ we have
\begin{equation}\label{pressureapriori}
\int\limits_{0}^{S_{lews}} \int\limits_{B_{\frac{3}{2}}(\bar{x})} |p-C_{\bar{x}}(t)|^{\frac{5}{3}} dxds\leq C(1+M^2)^\frac53.
\end{equation}

In order to take advantage of the fact that $u_0\in L^3(B_2(0))$, we decompose the initial data in the following way: $u_0=u_{0,a}+u_{0,b}$ with
\begin{equation*}
\supp u_{0,a}\subset B_\frac32(0),\quad u_{0,a}=u_0\ \mbox{on}\ B_1(0),\quad \nabla\cdot u_{0,a}=0,
\end{equation*}
and
\begin{equation*}
\|u_{0,a}\|_{L^3(\R^3)}\leq (1+K_3K_4)\|u_0\|_{L^3(B_2(0))},\quad \|u_{0,b}\|_{L^2_{uloc}(\R^3)}\leq (1+K_3K_4)\|u_0\|_{L^2_{uloc}(\R^3)},
\end{equation*}
where $K_3$ is a universal constant given by \eqref{e.cutoff} and $K_4$ by \eqref{e.bogo} below. 
This decomposition can be done in a standard way by cutting-off and using Bogovskii's operator \cite[Chapter III]{Galdibook}. Indeed, let $\chi\in C^\infty_c(\R^3)$ be a cut-off function such that
\begin{equation}\label{e.cutoff}
0\leq\chi\leq 1,\quad \supp\chi\subset B_2(0),\quad\chi=1\ \mbox{on}\ B_1(0),\quad |\nabla\chi|\leq K_3,
\end{equation}
where $K_3\in(0,\infty)$ is a universal constant. Then, we introduce $\tilde u_{0,a}$ given by Bogovskii's lemma, such that
\begin{align}\label{e.bogo}
\begin{split}
&\nabla\cdot\tilde u_{0,a}=u_0\cdot\nabla\chi,\quad \tilde u_{0,a}=0\ \mbox{on}\ \partial(B_2(0)\setminus B_1(0)),\\
&\|\tilde u_{0,a}\|_{L^3(B(0,2)\setminus B(0,1))}\leq K_4\|u_0\cdot\nabla\chi\|_{L^3(B_2(0)\setminus B_1(0))}\leq K_3K_4\|u_0\|_{L^3(B_2(0))},\\
&\|\tilde u_{0,a}\|_{L^2(B(0,2)\setminus B(0,1))}\leq K_4\|u_0\cdot\nabla\chi\|_{L^2(B_2(0)\setminus B_1(0))}\leq K_3K_4\|u_0\|_{L^2(B_2(0))},
\end{split}
\end{align}
where $K_4\in(0,\infty)$ is a universal constant. We extend $\tilde u_{0,a}$ by $0$ and let
\begin{equation*}
u_{0,a}=u_0\chi-\tilde u_{0,a}.
\end{equation*}
It is easy to check that this yields the decomposition above.

We now consider the unique mild solution $a$ associated to $u_{0,a}$. Existence of $a$ is recalled in Proposition \ref{prop.mild3}: we take  $\gamma_{univ}\in (0,\infty)$ accordingly. Without loss of generality, we assume that $\gamma_{univ}>0$ is taken sufficienty small such that
\begin{equation}\label{e.estgamma}
K_0(1+K_3K_4)\gamma_{univ}\leq \frac12\min(\ep,c^{-1}),
\end{equation}
where $K_0$ is given in Proposition \ref{prop.mild3}, $K_3$ in \eqref{e.cutoff}, $K_4$ in \eqref{e.bogo} and $c$ in \eqref{e.alhs}. Then 
\begin{equation}
\|a\|_{L^5_{t,x}}\leq \frac12\min(\ep,c^{-1})\quad\mbox{and}\quad \sup_{s\in(0,\infty)}s^\frac15\|a(\cdot,s)\|_{L^5}\leq\frac{\ep}{2}.
\end{equation}
Therefore, the drift $a$ satisfies the smallness conditions \eqref{e.epmorreya} and \eqref{smalldrift}. For proving Lemma \ref{lem.smallness} below, we also require estimates for the local energy and pressure for $a$. The mild solution $a$ satisfies $$\sup_{s\in(0,\infty)} s^{\frac{1}{8}}\|a(\cdot,s)\|_{L^{4}(\mathbb{R}^3)}\leq K_{0}\|u_{0,a}\|_{L^{3}(\mathbb{R}^3)}.$$
Thus, 
\begin{equation}\label{e.eqRiRj}
\|\mathcal{R}_{i}\mathcal{R}_{j}(a_{i}a_{j})\|_{L^{2}_{t,x}(\mathbb{R}^3\times (0,2))}\leq K_{0}^2\|u_{0,a}\|_{L^{3}(\mathbb{R}^3)}^2\,\,\,\textrm{and}
\end{equation}
$$ \|a-e^{t\Delta} u_{0,a}\|_{L^{\infty}_{t}L^{2}_{x}(\mathbb{R}^3\times (0,2))}^2+\|\nabla(a-e^{t\Delta} u_{0,a})\|_{L^{2}_{t}L^{2}_{x}(\mathbb{R}^3\times (0,2))}^2\leq K_{0}^2\|u_{0,a}\|_{L^{3}(\mathbb{R}^3)}^2. $$
Hence,
\begin{align}
\begin{split}\label{e.leia}
&\sup_{x} \big(\|a\|_{L^{\infty}_{t}L^{2}_{x}(B_{1}(x)\times (0,2))}+\|\nabla a\|_{L^{2}_{t}L^{2}_{x}(B_{1}(x)\times (0,2))}\big)\\
\leq\ & K_{0}^2\|u_{0,a}\|_{L^{3}(\mathbb{R}^3)}^2+ K_0\|u_{0,a}\|_{L^{3}(\mathbb{R}^3)}. 
\end{split}
\end{align}
We decompose $u=v+a$. The perturbation $v$ is a local energy solution to the perturbed Navier-Stokes system \eqref{e.pertns}. We have that $v(\cdot,0)|_{B_1(0)}=0$, hence we can get smallness of the local energy of $v$ for some short time $S^*(M)>0$. This is the purpose of the following lemma, which is the main ingredient for proving Theorem \ref {theo.main}. Such a result plays also a key r\^ole in the paper of Jia and \v{S}ver\'{a}k \cite[Theorem 3.1]{JS14}.

\begin{lemma}\label{lem.smallness}
There exists $S^*(M)\in(0,\frac14]$ such that 
\begin{equation}\label{e.locenEv}
\sup_{0<s<S^*(M)}\int\limits_{B_1(0)}|v(x,s)|^2dx+\int\limits_{B_1(0)\times (0,S^*(M))}|\nabla v|^2dxds\leq 1,
\end{equation}
and
\begin{equation}\label{e.epmorreyv}
\int\limits_{B_1(0)\times (0,S^*(M))}|v|^3+|q|^\frac32dxds\leq\ep_*.
\end{equation}
\end{lemma}

\begin{proof}[Proof of Lemma \ref{lem.smallness}]
Let $\phi\in C^\infty_c(\R^3)$ such that $0\leq\phi\leq 1$, $\supp\phi\subset B_\frac32(0)$, $\phi=1$ on $B_1(0)$, $|\nabla(\phi^2)|\leq K_5$ and $|\Delta(\phi^2)|\leq K_5'$ where $K_5,\ K_5'\in (0,\infty)$ are a universal constants. The local energy inequality then yields
\begin{align*}
y(t):=\ &\sup_{s\in(0,t)}\int\limits_{\R^3}|v(x,s)|^2\phi^2 dx+2\int\limits_{0}^t\int\limits_{\R^3}|\nabla v|^2\phi^2 dxds\\
\leq\ &\int\limits_{0}^t\int\limits_{\R^3}|v|^2\Delta(\phi^2) dxds+\int\limits_0^t\int\limits_{\R^3}|v|^2v\cdot\nabla(\phi^2) dxds+\int\limits_0^t\int\limits_{\R^3}2qv\cdot\nabla(\phi^2) dxds\\
&-\int\limits_0^t\int\limits_{\R^3}(a\cdot\nabla v)\cdot v\phi^2 dxds+\int\limits_0^t\int\limits_{\R^3}(a\otimes v):\nabla v \phi^2 dxds\\
&+\int\limits_0^t\int\limits_{\R^3}(a\otimes v):v\otimes\nabla(\phi^2) dxds\\
=\ &I_1+\ldots\ I_6,
\end{align*}
for all $t\in(0,S_{lews})$. From \eqref{e.aprioriu} and \eqref{e.leia}, we have that
\begin{align*}
\begin{split}
&\sup_{s\in(0,S_{lews})}\sup_{\bar x\in\R^3}\int\limits_{B_1(\bar x)}\frac{|v(x,s)|^2}{2}\, dx+\sup_{\bar x\in\R^3}\int\limits_0^{S_{lews}}\int\limits_{B_1(\bar x)}|\nabla v(x,s)|^2\, dx\, ds\\
\leq\ & K_1M^2+K_0^4(1+K_3K_4)^2\gamma_{univ}^2+K_0^2(1+K_3K_4)\gamma_{univ}\\
\leq\ &C(1+M^2),
\end{split}
\end{align*}
with $C\in(0,\infty)$ a universal constant. 
The terms $I_1$ and $I_2$ are energy subcritical, hence they decay in time:
\begin{align*}
|I_1|\leq\ & 
CK_5't^\frac25\Bigg(\int\limits_0^{S_{lews}}\int\limits_{B_2(0)}|v|^\frac{10}3dxds\Bigg)^\frac35\leq C(1+M^2)t^\frac25,\\
|I_2|\leq\ & 
CK_5t^\frac1{10}\Bigg(\int\limits_0^{S_{lews}}\int\limits_{B_2(0)}|v|^\frac{10}3dxds\Bigg)^\frac9{10}\leq C(1+M^3)t^\frac1{10},
\end{align*}
where $C\in(0,\infty)$ is a universal constant. 
The term $I_3$ requires to handle the pressure, so let us deal with it at the end of the proof. 
As for the other terms, we have
\begin{align}
|I_4|+|I_5|\leq\ & \|a\|_{L^5_{t,x}}\|\nabla v\phi\|_{L^2_{t,x}}\|v\phi\|_{L^\frac{10}3_{t,x}}\leq c\|a\|_{L^5}y(t)+C(1+M^2)t^\frac3{10},\label{e.alhs}\\
|I_6|\leq\ &CK_5t^\frac15\|a\|_{L^5_{t,x}}\|v\|_{L^\frac{10}3_{t,x}(B_2(0)\times(0,S_{lews}))}^2\leq C(1+M^2)t^\frac15,\nonumber
\end{align}
where $c,\ C\in(0,\infty)$ are universal constants. 
We now turn to the term $I_3$ involving the pressure. Thanks to \eqref{pressureapriori} and \eqref{e.eqRiRj} we can estimate the pressure $q$ and obtain
\begin{equation*}
\int\limits_{0}^{S_{lews}} \int\limits_{B_{\frac{3}{2}}(\bar{x})} |q-C_{\bar{x}}(t)|^{\frac{5}{3}} dxds\leq C\big((1+M^2)^\frac53+\|u_{0,a}\|_{L^3}^4\big).
\end{equation*}
Then, we get the bound,
\begin{equation*}
|I_3|\leq 
C(1+M^3)(t^\frac1{10}+t^\frac7{10}).
\end{equation*}
We finally obtain, for all $t\in (0,S_{lews})$,
\begin{equation}\label{e.smallen}
\frac{y(t)}{2}\leq(1-c\|a\|_{L^5_{t,x}})y(t)\leq C_*(1+M^3)t^\frac1{10},
\end{equation}
where $C_*\in(0,\infty)$ is a universal constant. 
Moreover, 
\begin{align*}
&\int\limits_{B_1(0)\times (0,t)}|v|^3+|q|^\frac32dxds\\
\leq\ &C_*'(t^\frac1{10}(y(t))^\frac32+M^2t)\\
\leq\ &C_*'((2C_*(1+M^3))^\frac32+1+M^2)t^\frac1{4},
\end{align*}
where $C_*'\in(0,\infty)$ is a universal constant. 
Choosing $S^*(M)\in(0,\frac14]$ such that
\begin{equation*}
2C_*(1+M^3)(S^*(M))^\frac1{10}\leq 1\quad\mbox{and}\quad C_*'((2C_*(1+M^3))^\frac32+1+M^2)(S^*(M))^\frac1{4}\leq\ep_*
\end{equation*}
concludes the proof of the lemma.
\end{proof}

To conclude the proof of Theorem \ref{theo.main}, it remains to extend $v$ and $a$ by zero on the time interval $(-1+S^*(M),0)$. Estimate \eqref{e.smallen} implies the strong convergence of the local energy to zero when $t\to 0$. Hence the extended $v$ is a local suitable solution to \eqref{e.pertns}, which satisfies the bounds \eqref{e.locenE} with $E=1$ and \eqref{e.epmorrey}. Therefore, it results from Theorem \ref{theo.morrey} that for all $(x,t)\in\bar Q_\frac12(0,S^*(M))$, for all $r\in(0,\frac14]$, we have the subcritical Morrey bound
\begin{equation*}
\frac{1}{r^4}\int\limits_{Q_r(\bar x,t)}|v|^3dxds\leq C_*\ep_*^\frac23,
\end{equation*}
where we recall that we have taken $\delta=1$ at the beginning of this section. We finally directly apply Theorem \ref{localboundednesscriticalStokes} to get that $v\in L^\infty(B_\frac13(0)\times(0,S^*(M)))$. This concludes the proof.

\subsection{Proof of Theorem \ref{theo.conc}}

The proof of Theorem \ref{theo.conc} is by contradiction. Let $\gamma_{univ},\, M,\,  S^*(M)$ be fixed as in Theorem \ref{theo.main}. Let $T^*\in (0,\infty)$ and $r_0(0,\infty]$ be fixed. Let 
\begin{equation}
t_*(T^*,M,r_0):=T^*-S^*(M)r_0.
\end{equation}
Assume that there exists a Leray-Hopf solution $u$ to \eqref{e.nse} satisfying the type I bound \eqref{e.typeI} such that $u$ blows-up at $(0,T^*)$ and there exists  $t_0\in(t_*(M,r_0,T^*),T^*)$ such that
\begin{equation*}
\|u(\cdot,t_0)\|_{L^3\big(|\cdot|\leq 2\sqrt{\frac{T^*-t_0}{S^*(M)}}\big)}\leq\gamma_{univ}.
\end{equation*}
We then rescale $u$ with the parameter 
\begin{equation*}
\lambda:=\sqrt{\frac{T^*-t_0}{S^*(M)}}
\end{equation*}
according to the scaling of the Navier-Stokes equations: for all $(y,s)\in \R^3\times (0,\infty)$,
\begin{equation*}
u_\lambda(y,s):=\lambda u(\lambda y,t_0+\lambda^2s).
\end{equation*}
Then, we have $\|u_\lambda(\cdot,0)\|_{L^3(|\cdot|\leq 2)}\leq\gamma_{univ}$. Furthermore, since $S^*(M)\leq 1$ and by our choice of $t_*(T^*,M,r_0)$, we can take $r=\lambda\leq r_0$ and $t=t_0$ in \eqref{e.typeI}. Therefore, we obtain that $\|u_\lambda(\cdot,0)\|_{L^2_{uloc}(\R^3)}\leq M$ . Theorem \ref{theo.main} enables to conclude that $u_\lambda$ is regular at the space-time point $(0,S^*(M))$, which is a contradiction.

\appendix

\section{Auxiliary results}
\label{app.a}

The first result is the classical existence of mild solutions for critical initial data $a\in L^3_\sigma(\R^3)$. 

\begin{proposition}[\cite{Giga86}
]\label{prop.mild3}
There exists  universal constants $\gamma,\, K_0\in (0,\infty)$ such that the following holds true. For for all $u_{0,a}\in L^3_\sigma(\R^3)$ with $\|u_{0,a}\|_{L^3}\leq\gamma$, there exists a unique smooth mild solution $a\in C([0,\infty);L^3)\cap L^\infty((0,\infty);L^3)$ such that $a(\cdot,0)=u_{0,a}$, $a\in C((0,\infty);W^{1,3}\cap BUC_\sigma)$ and
\begin{multline}\label{e.boundmild}
\sup_{t\in(0,\infty)}\big(\|a(\cdot,t)\|_{L^3}+t^{\frac{1}{8}}\|a(\cdot,t)\|_{L^{4}}+t^\frac15\|a(\cdot,t)\|_{L^5}+t^\frac12\|a(\cdot,t)\|_{L^\infty}+t^\frac12\|\nabla a(\cdot,t)\|_{L^3}\big)\\
\leq K_0\|u_{0,a}\|_{L^3}.
\end{multline}
Moreover $a\in L^5(\R^3\times (0,\infty))$ and
\begin{equation}\label{e.L5tx}
\|a\|_{L^5(\R^3\times (0,\infty))}\leq K_0\|u_{0,a}\|_{L^3}.
\end{equation}
\end{proposition}

The second result is an estimate for local energy solutions, so-called Lemari\'e-Rieusset solutions \cite[Chapter 32 and 33]{LR02}. Before stating this result, we give the definition of such solutions. See also \cite{KS07} and \cite[Definition 3.1]{JS14}.

\begin{definition}[Local energy solutions]\label{def.lews}
A pair $(u,p)$ is called a local energy solution to \eqref{e.nse} in $\R^3\times (0,\infty)$ with the initial data $u_0\in L^2_{uloc,\sigma}(\R^3)$, $\sup_{|\bar x|\geq R}\|u_0\|_{L^2(B_1(\bar x))}\stackrel{R\rightarrow\infty}{\longrightarrow} 0$, if $(u,p)$ satisfies the following conditions:

\noindent {\rm (i)} We have $u\in L^\infty_{loc}([0,T); L^2_{uloc,\sigma} (\R^3))$ and 
$p\in L^\frac32_{loc} (\R^3\times (0,\infty))$,  and 
\begin{align*}
\sup_{\bar x\in \R^3} \int\limits_0^{T} \| \nabla u \|_{L^2 (B_1(\bar x)\cap \R^3)}^2 d t <\infty,
\end{align*}
for all finite $T\in (0, \infty)$.

\noindent {\rm (ii)} The pair $(u,p)$ is a solution to \eqref{e.nse} in the sense of distributions.

\noindent {\rm (iii)} The function $t\mapsto \langle u(\cdot,t), \varphi \rangle_{L^2(\R^3)}$ belongs to $C([0,T))$ for any compactly supported $\varphi\in L^2(\R^3)$. Moreover, for any compact set $K\subset \R^3$,
\begin{align}\label{e.cvinit}
\lim_{t\rightarrow 0} \| u(\cdot,t) -u_0 \|_{L^2(K)} =0.
\end{align} 

\noindent {\rm (iv)} The pair $(u,p)$ satisfies the local energy inequality: for any $0\leq \phi\in C^\infty_c(\R^3\times (0,\infty))$, for all $t\in (0,\infty)$,
\begin{align*}
\begin{split}
&\int\limits_{B_1(0)}|u(x,t)|^2\phi(x,t) dx+2\int\limits_{0}^t\int\limits_{\R^3}|\nabla u|^2\phi dxds\\
\leq\ &\int\limits_{0}^t\int\limits_{\R^3}|u|^2(\partial_t\phi+\Delta\phi)dxds+\int\limits_{0}^t\int\limits_{\R^3}(|u|^2+2p)u\cdot\nabla\phi dxds.
\end{split}
\end{align*}
\end{definition}

Notice that \eqref{e.cvinit} enables to transfer the mild decay of the initial data to the solution $u$. 

\begin{proposition}[{\cite[Lemma 3.1]{JS14}, \cite{LR02}}]\label{prop.lews}
There exist two universal constants $c_1,\ K_1\in (0,\infty)$ such that for all $u_0\in L^2_{uloc}(\R^3)$, for all local energy solution $u$ to \eqref{e.nse} with initial data $u_0$, we have
\begin{equation}\label{e.apriori}
\sup_{s\in(0,S_{lews})}\sup_{\bar x\in\R^3}\int\limits_{B_1(\bar x)}\frac{|u(x,s)|^2}{2}\, dx+\sup_{\bar x\in\R^3}\int\limits_0^{S_{lews}}\int\limits_{B_1(\bar x)}|\nabla u(x,s)|^2\, dx\, ds\leq K_1M^2
\end{equation}
where 
\begin{equation*}
S_{lews}:=c_1^2\min(M^{-4},1)\quad\mbox{and}\quad M:=\Bigg(\sup_{\bar x\in\R^3}\int\limits_{B_1(\bar x)}|u_0(x)|^2\, dx\Bigg)^\frac12<\infty. 
\end{equation*}
\end{proposition}

Following \cite{LR02,KS07,JS13}, if $u$ is a local energy weak solution to \eqref{e.nse} in the sense of Definition \ref{def.lews} and $a$ is a mild solution to \eqref{e.nse}, then $v-a$ solves
\begin{equation*}
\partial_{t}v-\Delta v+v\cdot\nabla v+ a\cdot\nabla v+v\cdot\nabla a+ \nabla q=0,\,\,\,\,\nabla\cdot v=0,\quad x\in\R^3,\ t>0.
\end{equation*}
in the sense of distributions and we have the following \emph{global} representation formula for the pressure: 
for all $\bar x\in\R^3$, for all $(x,t)\in B_\frac32(\bar x)\times (0,T)$,
\begin{align}\label{e.forpre}
\begin{split}
q(x,t)=\ &-\frac13|v(x,t)|^2-\int\limits_{B_2(\bar x)}\nabla^2N(x-y):(v\otimes v+a\otimes v+v\otimes a)dy\\
&-\int\limits_{B_2(\bar x)}(\nabla^2N(x-y)-\nabla^2N(\bar x-y)):(v\otimes v+a\otimes v+v\otimes a)dy.
\end{split}
\end{align}
Here $N(x)=-\frac{1}{4\pi |x|}$. Notice that \eqref{e.forpre} provides a proof of estimate \eqref{pressureapriori} in the case $a=0$. With $\bar x=0$, we have for all $(x,t)\in B_\frac32(0)\times(0,S_{lews})$,
\begin{equation*}
q(x,t):=-\frac13|v(x,t)|^2+q_{loc}(x,t)+q_{nonloc}(x,t),
\end{equation*}
with  
\begin{align*}
\|q_{loc}(\cdot,t)\|_{L^\frac53_{x,t}(B_\frac32(0)\times  S_{lews})}\leq\ & C\|v^2\|_{L^\frac53_{t,x}}\leq C(1+M^2),\\
\|q_{nonloc}(\cdot,t)\|_{L^\infty(B_\frac32(0))}\leq\ &C\sum_{\xi\in\Z^3}\frac1{1+|\xi|^4}\int\limits_{B_1(\xi)}|v|^2\leq C(1+M^2),
\end{align*}
where we used Calder\'on-Zygmund estimates for the first bound. 

We conclude this appendix by a useful \emph{local} representation formula for the pressure. The following lemma is well known, see for instance Caffarelli, Kohn and Nirenberg's paper \cite{CKN82} (specifically p.782 of \cite{CKN82}).

\begin{lemma}\label{pressure localisation}
Suppose that $p\in L^{1}(B_{2}(0))$ and $V_{ij}\in L^{q}(B_{2}(0))$ for some $q>1$. Furthermore, suppose
\begin{equation}\label{pressureequation}
-\Delta p= \partial_{i}\partial_{j} V_{ij}
\end{equation}
in $B_{2}(0)$ in a distributional sense.
Then for $\varphi\in C_{0}^{\infty}(B_{1}(0))$ we have
\begin{equation}\label{localisedpressuresplitting}
\varphi p= \mathcal{R}_{i}\mathcal{R}_{j}(\varphi V_{ij})-N\ast((\partial_{i}\partial_{j}\varphi)V_{ij})-2\partial_{j}N\ast((\partial_{i}\varphi)V_{ij})-N\ast(p\Delta\varphi)-2\partial_{j}N\ast((\partial_{j}\varphi)p).
\end{equation}
\end{lemma}

\section{The case of $L^{3,\infty}$ initial data}
\label{app.b}
\subsection{Preliminary material}
Given a measurable subset $\Omega\subseteq\mathbb{R}^{d}$, let us define the Lorentz spaces. 
For a measurable function $f:\Omega\rightarrow\mathbb{R}$ define:
\begin{equation}\label{defdistchapter2}
d_{f,\Omega}(\alpha):=\mu(\{x\in \Omega : |f(x)|>\alpha\}),
\end{equation}
where $\mu$ denotes the Lebesgue measure.
 The Lorentz space $L^{p,q}(\Omega)$, with $p\in [1,\infty[$, $q\in [1,\infty]$, is the set of all measurable functions $g$ on $\Omega$ such that the quasinorm $\|g\|_{L^{p,q}(\Omega)}$ is finite. Here:

\begin{equation}\label{Lorentznormchapter2}
\|g\|_{L^{p,q}(\Omega)}:= \Big(p\int\limits_{0}^{\infty}\alpha^{q}d_{g,\Omega}(\alpha)^{\frac{q}{p}}\frac{d\alpha}{\alpha}\Big)^{\frac{1}{q}},
\end{equation}
\begin{equation}\label{Lorentznorminftychapter2}
\|g\|_{L^{p,\infty}(\Omega)}:= \sup_{\alpha>0}\alpha d_{g,\Omega}(\alpha)^{\frac{1}{p}}.
\end{equation}\\
It is known there exists a norm, which is equivalent to the quasinorm defined above, for which $L^{p,q}(\Omega)$ is a Banach space. 
For $p\in [1,\infty)$ and $1\leq q_{1}< q_{2}\leq \infty$, we have the following continuous embeddings 
\begin{equation}\label{Lorentzcontinuousembeddingchapter2}
L^{p,q_1}(\Omega) \hookrightarrow  L^{p,q_2}(\Omega)
\end{equation}
and the inclusion is known to be strict.

Let $X$ be a Banach space with norm $\|\cdot\|_{X}$, $ a<b$, $p\in [1,\infty)$ and  $q\in [1,\infty]$. 
 Then $L^{p,q}(a,b;X)$   will denote the space of strongly measurable $X$-valued functions $f(t)$ on $(a,b)$ such that 

\begin{equation}\label{Lorentznormbochnerchapter2}
\|f\|_{L^{p,q}(a,b; X)}:= \big\|\|f(t)\|_{X}\big\|_{L^{p,q}(a,b)}<\infty.
\end{equation}
In particular, if $1\leq q_{1}< q_{2}\leq \infty$, we have the following continuous embeddings 
\begin{equation}\label{Bochnerlorentzcontinuousembeddingchapter2}
L^{p,q_1}(a,b; X) \hookrightarrow  L^{p,q_2}(a,b; X)
\end{equation}
and the inclusion is known to be strict.

Let us recall a known proposition known as `O'Neil's convolution inequality' (Theorem 2.6 of  O'Neil's paper \cite{O'Neil}).
\begin{proposition}\label{O'Neilchapter2}
Suppose $1< p_{1}, p_{2}, r<\infty$ and $1\leq q_{1}, q_{2}, s\leq\infty $ are such that
\begin{equation}\label{O'Neilindices1chapter2}
\frac{1}{r}+1=\frac{1}{p_1}+\frac{1}{p_{2}}
\end{equation}
and
\begin{equation}\label{O'Neilindices2chapter2}
\frac{1}{q_1}+\frac{1}{q_{2}}\geq \frac{1}{s}.
\end{equation}
Suppose that
\begin{equation}\label{fghypothesischapter2}
f\in L^{p_1,q_1}(\mathbb{R}^{d})\,\,\textrm{and}\,\,g\in  L^{p_2,q_2}(\mathbb{R}^{d}).
\end{equation}
Then
\begin{equation}\label{fstargconclusion1chapter2}
f\ast g \in L^{r,s}(\mathbb{R}^d)\,\,\rm{with} 
\end{equation}
\begin{equation}\label{fstargconclusion2chapter2}
\|f\ast g \|_{L^{r,s}(\mathbb{R}^d)}\leq 3r \|f\|_{L^{p_1,q_1}(\mathbb{R}^d)} \|g\|_{L^{p_2,q_2}(\mathbb{R}^d)}. 
\end{equation}
\end{proposition}
We will use an inequality that we will refer to as `Hunt's inequality'. The statement below and proof can be found in Hunt's paper \cite{hunt} (Theorem 4.5, p.271 of \cite{hunt}).
\begin{proposition}\label{hunt}
Suppose that $0<p,q,r\leq\infty$ and $0<s_1,s_2\leq\infty$. Furthermore, suppose that $p$, $q$, $r$, $s_1$ and $s_2$ satisfy the following relations: $$\frac{1}{p}+\frac{1}{q}=\frac{1}{r}$$ and $$\frac{1}{s_1}+\frac{1}{s_2}=\frac{1}{s}.$$ 
Then the assumption that $f\in L^{p,s_1}(\Omega)$ and $g\in L^{q,s_2}(\Omega)$ implies  that $fg \in L^{r,s}(\Omega)$, with the estimate 
\begin{equation}\label{Holderverygeneral}
\|fg\|_{L^{r,s}(\Omega)}\leq C(p,q,s_1,s_2)\|f\|_{L^{p,s_1}(\Omega)}\|g\|_{L^{q,s_2}(\Omega)}.
\end{equation}
\end{proposition}
As a result of the above Propositions, we have the following estimates with $B_{R}(0)\subset \mathbb{R}^3$ (which we will frequently use):
\begin{align}\label{criticalLorentzestimate}
\begin{split}
&\int\limits_{Q_{R}(0)} |\nabla f| |f| |g| dxdt\\
\leq\ &\|g\|_{L^{\infty}_{t}L^{3,\infty}_{x}(Q_{R}(0))}\Bigg( \|\nabla f\|_{L^{2}(Q_{R}(0))}^2+ \frac{1}{R}\int\limits_{-R^2}^{0}\|\nabla f\|_{L^{2}(B_{R}(0))}\|f\|_{L^{2}(B_{R}(0))}dt\Bigg),
\end{split}
\end{align}
\begin{equation}\label{lowerordertermestimate}
\int\limits_{Q_{R}(0)}  |f|^2 |g| dxdt\leq  R^{\frac{2}{3}}\|g\|_{L^{5,\infty}_{t}L^{5}_{x}(Q_{R}(0))} \| f\|_{L^{3}(Q_{R}(0))}^2.
\end{equation}
The first estimate is stated and proven in \cite{dubois}, for example.
Now, we state known results for the Navier-Stokes equations with initial data in $L^{3,\infty}(\mathbb{R}^3).$ We refer the reader to \cite{meyer} and \cite{Plan96}.
\begin{proposition}
\label{prop.mildweak3}
There exists universal constants $\gamma,\, K_0'\in (0,\infty)$ such that the following holds true. For all $u_{0,a}\in L^{3,\infty}_\sigma(\R^3)$, $\|u_{0,a}\|_{L^{3,\infty}}\leq\gamma$, there exists a smooth mild solution $a\in   C_{w^{*}}([0,\infty);L^{3,\infty})\cap L^\infty((0,\infty);L^{3,\infty})$ such that $a(\cdot,0)=u_{0,a}$ and
\begin{equation}\label{e.boundmildweakL3}
\sup_{t\in(0,\infty)}\big(\|a(\cdot,t)\|_{L^{3,\infty}}+t^\frac18\|a(\cdot,t)\|_{L^4}+t^\frac15\|a(\cdot,t)\|_{L^5}+t^\frac12\|a(\cdot,t)\|_{L^\infty})\leq K_0'\|u_0\|_{L^{3,\infty}}.
\end{equation}
The mild solution is unique in the class of solutions with small enough $L^\infty(0,\infty; L^{3,\infty})$ norm.
\end{proposition}
\subsection{$L^{3,\infty}$ initial data: Section \ref{sec.morrey}}
We briefly describe the changes the required for Section \ref{sec.morrey}. With the above Proposition in mind concerning mild solutions, in Section \ref{sec.morrey} we can no longer assume $a$ is in $L^{5}_{x,t}$. Instead we assume
\begin{equation}\label{step2weakL3smallness}
\|a\|_{L^{\infty}(-1,0; L^{3,\infty}(B_{1}(0)))}+\|a\|_{L^{5,\infty}(-1,0; L^{5}(B_{1}(0)))}\leq \varepsilon_{*}. 
\end{equation}
The first adjustment regards the estimate of the pressure (Lemma \ref{lem.pressure}). In particular, Hunt's inequality can be used to show that the second and last term in \eqref{e.estpress} can be replaced by
$$C_{2}r^{\frac{1-\delta}{2}}\Bigg(\int\limits_{Q_{2r}(0,0)}|v|^3 dxds\Bigg)^{\frac{1}{2}}\|a\|_{L^{5,\infty}_{t}L^{5}_{x}(Q_{1}(0,0))}^{\frac{3}{2}} $$
and
$$C_{2}r^{\frac{44-5\delta}{10}}\rho^{\frac{-39}{10}}\Bigg(\int\limits_{Q_{2r}(0,0)}|v|^3 dxds\Bigg)^{\frac{1}{2}}\|a\|_{L^{5,\infty}_{t}L^{5}_{x}(Q_{1}(0,0))}^{\frac{3}{2}}.  $$
Now we proceed to the adjustments needed for the proof of Theorem \ref{theo.morrey}.
In Step 2 and Step 3 the only adjustment is to make extensive use of \eqref{criticalLorentzestimate}-\eqref{lowerordertermestimate}.
In Step 4 we take the adjustment of Lemma \ref{lem.pressure} into account. Moreover, when estimating the pressure we have to use Hunt's inequality to estimate $J_{4}$. In particular, this gives
$$J_{4}\leq C\left(\sum_{k=2}^{n-1}r_{k+1}^{-4+\frac9{10}}\big(\ep_*^\frac23r_{k}^{3-\frac23\delta}\big)^\frac12\right)^\frac32\|a\|_{L^{5,\infty}_{t}L^{5}_{x}(Q_{1}(0,0))}^\frac32r_{n+1}^\frac75 .$$
\subsection{$L^{3,\infty}$ initial data: Section \ref{sec.boundedness}}
As in Section \ref{sec.morrey} we can no longer assume $a$ is in $L^{5}_{x,t}$. Instead we assume
\begin{equation}\label{step3weakL3smallness}
\sup_{0<s<\infty} s^{\frac{1}{5}}\|a(\cdot,s)\|_{L_{5}(\mathbb{R}^3)}\leq \ep_{**},
\end{equation}
where $\ep_{**}>0$ is some small universal constant. The only difference in Section \ref{sec.boundedness} regards \eqref{criticalestfinite}. In particular, we can use Young's inequality in space, followed by O'Neil's convolution inequality in time and finally Hunt's inequality in time  to see that the following holds: for $1<q<\infty$,
\begin{align*}
\|L(\nabla\cdot(a\otimes b))\|_{L^{q}_{x,t}(\mathbb{R}^3\times (0,T))}\leq\ & C(q)\big\|\|a(\cdot,s)\|_{L^{5}_{x}}\|b(\cdot, s)\|_{L^{q}_{x}}\big\|_{L^{\frac{5q}{5+q},q}(0,T)}\\
\leq\ & C'(q)\sup_{0<s<\infty} s^{\frac{1}{5}}\|a(\cdot,s)\|_{L^{5}(\mathbb{R}^3)} \|b\|_{L^{q}_{x,t}(\mathbb{R}^3\times (0,T))}.
\end{align*}

\subsection{$L^{3,\infty}$ initial data: Section \ref{sec.proof}}
First \eqref{e.bogo} must be adjusted using Hunt's inequality as follows:
\begin{align}\label{e.bogo.weakL3}
\begin{split}
\|\tilde u_{0,a}\|_{L^{3,\infty}(B_{2}(0)\setminus B_{1}(0))}\leq\ & C
\|\nabla \tilde u_{0,a} \|_{L^2(B_2(0)\setminus B_1(0))}\\
\leq\ & C(\chi)
\|u_0\|_{L^2(B_{2}(0))}\leq C(\chi)\|u_{0}\|_{L^{3,\infty}(B_{2}(0))}.
\end{split}
\end{align}
Then in \eqref{e.alhs}, we must instead make use of \eqref{criticalLorentzestimate}-\eqref{lowerordertermestimate}.
\section{The case of Besov initial data}
\label{app.c}
\subsection{Preliminaries}
Let $d, m \in \N\setminus\{0\}$. We begin by recalling the definition of the \emph{homogeneous Besov spaces} $\dot B^s_{p,q}(\R^d;\R^m)$.  There exists a non-negative radial function $\varphi \in C^\infty(\R^d)$ supported on the annulus $\{ \xi \in \R^d : 3/4 \leq |\xi| \leq 8/3 \}$ such that
\begin{equation}
	\sum_{j \in \Z} \varphi(2^{-j} \xi) = 1, \quad \xi \in \R^3 \setminus \{0\}.
	\label{}
\end{equation}
The homogeneous Littlewood-Paley projectors $\dot \Delta_j$ are defined by
\begin{equation}
	\dot \Delta_j f = \varphi(2^{-j} D) f, \quad j \in \Z,
	\label{}
\end{equation}
for all tempered distributions $f$ on $\R^d$ with values in $\R^m$. The notation $\varphi(2^{-j}D) f$ denotes convolution with the inverse Fourier transform of $\varphi(2^{-j}\cdot)$ with $f$.

 Let $p,q \in [1,\infty]$ and $s \in (-\infty,d/p)$.\footnote{The choice $s=d/p$, $q=1$ is also valid.} The homogeneous Besov space $\dot B^s_{p,q}(\R^d;\R^m)$ consists of all tempered distributions $f$ on $\R^d$ with values in $\R^m$ satisfying
	\begin{equation}
		\|{f}\|_{\dot B^s_{p,q}(\R^d;\R^m)} :=\Big(\sum_{j\in\mathbb{Z}} \big( 2^{js}\|\dot{\Delta}_{j} f\|_{L^{p}}\big)^{q}\Big)^{\frac{1}{q}}.
		\label{}
	\end{equation}
	and such that $\sum_{j \in \Z} \dot \Delta_j f$ converges to $f$ in the sense of tempered distributions on $\R^d$ with values in $\R^m$. In this range of indices, $\dot B^{s}_{p,q}(\R^d;\R^m)$ is a Banach space. When $s \geq 3/p$ and $q > 1$, the spaces must be considered \emph{modulo polynomials}. Note that other reasonable choices of the function $\varphi$ defining $\dot \Delta_j$ lead to equivalent norms. 

We now recall a particularly useful property of Besov spaces, i.e., their characterization in terms of the heat kernel. For all $s \in (-\infty,0)$, there exists a constant $c := c(s) > 0$ such that for all tempered distributions $f$ on $\R^3$,
		\begin{equation}
			c^{-1} \sup_{t > 0} t^{-\frac{s}{2}} \|e^{t\Delta} f\|_{L^p(\R^3)} \leq \|{f}\|_{\dot B^s_{p,\infty}(\R^3)} \leq c \sup_{t > 0} t^{-\frac{s}{2}}\|{e^{t\Delta} f}\|_{L^p(\R^3)}.
			\label{besovequivalentnorm}
		\end{equation}
Let $\Omega\subset\mathbb{R}^3$ be a domain with sufficiently smooth boundary. We say $u\in \dot{B}^{s}_{p,q}(\Omega)$ if 
\begin{enumerate}[label=(*)]
\item\label{extprop} (\textit{extension property}) there exists $E(u)\in \dot{B}^{s}_{p,q}(\mathbb{R}^3)$ such that $E(u)=u$ on $\Omega$ as distributions.
\end{enumerate}
Then 
$$\|u\|_{\dot{B}^{s}_{p,q}(\Omega)}:=\inf{\{\|E(u)\|_{\dot{B}^{s}_{p,q}(\mathbb{R}^3)}:\,\,\,\,\,E(u)\,\,\,\,\,\textrm{satisfies}\,\,\,\,\,\ref{extprop}\}}.
$$
In what follows, we will mostly use just one feature of the definition of Besov spaces on bounded domains:
\begin{equation}\label{Besovboundedcutoff}
u_{0}\in \dot{B}^{-1+\frac{3}{p}}_{p,\infty}(B_{2}(0)),\,\,\,\,\,\,\varphi\in C_{0}^{\infty}(B_{2}(0)) \Rightarrow  \|\varphi u_{0}\|_{B^{-1+\frac{3}{p}}_{p,\infty}(\mathbb{R}^3)}\leq C(\varphi) \|u_{0}\|_{\dot{B}^{-1+\frac{3}{p}}_{p,\infty}(B_{2}(0))}.
\end{equation}
The proof of this uses the definition of Besov spaces on bounded domains and the fact that for $\varphi$ in the Schwartz class $$\|f\varphi\|_{B^{-1+\frac{3}{p}}_{p,\infty}(\mathbb{R}^3)}\leq C(\varphi)\|f\|_{\dot{B}^{-1+\frac{3}{p}}_{p,\infty}(\mathbb{R}^3)}.$$
The proof of this is along the lines of Proposition 2.3 of \cite{LiWang2018blowup}.

We will also make use of a decomposition result for Homogeneous Besov spaces. The statement without (\ref{L2persistency}) can be found in \cite{AlbrittonBarkerBesov2018}. See also \cite{barker2017existence}.
\begin{lemma}\label{splitlem}
Let $p \in (3,\infty)$. There exist  $\gamma_1, \gamma_2 > 0$, and $C > 0$, each depending only on $p$, such that for each divergence-free vector field $g \in \dot{B}^{-1+\frac{3}{p}}_{p,\infty}\cap L^{2}(\mathbb{R}^3)$ and $N>0$, there exist divergence-free vector fields $\bar{g}^{N}\in \dot{B}^{-1+\delta_{2}}_{\infty,\infty}(\mathbb{R}^3)\cap \dot{B}^{-1+\frac3p}_{p,\infty}(\mathbb{R}^3)\cap L^{2}(\mathbb{R}^3)$ and $\widetilde{g}^{N}\in L^2(\mathbb{R}^3)\cap \dot{B}^{-1+\frac3p}_{p,\infty}(\mathbb{R}^3)$ with the following properties:
\begin{equation}\label{Vdecomp1}
g= \widetilde{g}^{N} + \bar{g}^{N},
\end{equation}
\begin{equation}\label{barV_2est}
\|\widetilde{g}^{N}\|_{L^2(\R^3)}\leq C N^{-\gamma_{2}} \|{g}\|_{\dot{B}^{-1+\frac{3}{p}}_{p,\infty}},
\end{equation} 
\begin{equation}\label{barV_1est}
\|\bar{g}^{N}\|_{\dot{B}^{-1+\delta_{2}}_{\infty,\infty}(\R^3)}\leq C N^{\gamma_{1}} \|{g}\|_{\dot{B}^{-1+\frac{3}{p}}_{p,\infty}}.
\end{equation}
Furthermore, 
\begin{equation}\label{barV_1est.1}
\|\widetilde{g}^{N}\|_{\dot{B}^{-1+\frac{3}{p}}_{p,\infty}(\R^3)}, \|\bar{g}^{N}\|_{\dot{B}^{-1+\frac{3}{p}}_{p,\infty}(\R^3)}\leq C \|{g}\|_{\dot{B}^{-1+\frac{3}{p}}_{p,\infty}}
\end{equation}
and
\begin{equation}\label{L2persistency}
\|\widetilde{g}^{N}\|_{L^{2}(\R^3)}, \|\bar{g}^{N}\|_{L^{2}(\R^3)}\leq C \|{g}\|_{L^{2}}.
\end{equation}
\end{lemma}
Finally, we state known results for the Navier-Stokes equations with data in $\dot{B}^{-1+\frac{3}{p}}_{p,\infty}(\mathbb{R}^3).$ We refer the reader to  \cite{Plan96}, for example.
\begin{proposition}
\label{prop.mildBesov}
Let $S_{mild}\in(0,\infty)$ and $p\in(3,\infty)$. There exists two constants $\gamma(p)\in (0,\infty)$ and $K_0''(p)\in (0,\infty)$ such that the following holds true. For all divergence-free $u_{0,a}\in \dot{B}^{-1+\frac{3}{p}}_{p,\infty}(\mathbb{R}^3)$, 
$$\sup_{0<t<S_{mild}} t^{\frac{1}{2}(1-\frac{3}{p})}\|e^{t\Delta}u_{0,a}\|_{L^{p}}\leq\gamma(p),$$ 
there exists a smooth mild solution $a\in   C_{w^{*}}([0,S_{mild});\dot{B}^{-1+\frac{3}{p}}_{p,\infty})\cap L^\infty((0,S_{mild});\dot{B}^{-1+\frac{3}{p}}_{p,\infty})$ such that $a(\cdot,0)=u_{0,a}$ and
\begin{multline}\label{e.boundmildBesov}
\sup_{t\in(0,S_{mild})}\big(\|a(\cdot,t)\|_{\dot{B}^{-1+\frac{3}{p}}_{p,\infty}}+ t^{\frac{1}{2}(1-\frac{3}{p})}\|a(\cdot,t)\|_{L^{p}}+t^\frac12\|a(\cdot,t)\|_{L^\infty}\big)\\
\leq K_0''(p)\sup_{t\in(0,S_{mild})} t^{\frac{1}{2}(1-\frac{3}{p})}\|e^{t\Delta}u_{0,a}\|_{L^{p}} .
\end{multline}
The mild solution is unique in the class of solutions with sufficiently small 
$$\sup_{t\in(0,S_{mild})}t^{\frac{1}{2}(1-\frac{3}{p})}\|a(\cdot,t)\|_{L^{p}}$$ 
norm.
\end{proposition}
\subsection{Besov initial data: Section \ref{sec.boundedness}}
In this section we should now assume
$$\sup_{0<s<T} s^{\frac{1}{2}(1-\frac{3}{p})}\|a(\cdot,s)\|_{L^{p}(\mathbb{R}^3)}\leq \varepsilon_{**},$$
where $\ep_{**}>0$ is some small universal constant. With this adjustment, the arguments in Section \ref{sec.boundedness} are the same as in the case of $L^{3,\infty}$ initial data.
\subsection{Besov initial data: Section \ref{sec.proof}}
\subsubsection{The extension operator}
Throughout this part we assume 
\begin{equation*}\label{Besovbound}
\|u_{0}\|_{\dot{B}^{-1+\frac{3}{p}}_{p,\infty}(B_{2}(0))}\leq\gamma_{univ}\,\,\,\,\,\,\,\textrm{and}\,\,\,\,\,\,\,\,\|u_{0}\|_{L^{2}(B_{2}(0))}\leq M.
\end{equation*}
For convenience, we assume without loss of generality that $p\in(6,\infty)$.
Let $\chi\in C^\infty_c(\R^3)$ be a cut-off function such that 
\begin{equation*}\label{e.cutoffBesov}
0\leq\chi\leq 1,\quad \supp\chi\subset B_2(0),\quad\chi=1\ \mbox{on}\ B_\frac{3}{2}(0),\quad |\nabla\chi|\leq K_3,
\end{equation*}
where $K_3\in(0,\infty)$ is a universal constant.
From the preliminaries we have
\begin{equation}\label{cutoffBesovestimate}
\|\chi u_{0}\|_{\dot{B}^{-1+\frac{3}{p}}_{p,\infty}(\mathbb{R}^3)}\leq C(\chi)\gamma_{univ}\leq \gamma{'}_{univ}.
\end{equation}
Obviously, $$\|\chi u_{0}\|_{L^{2}(\mathbb{R}^3)}\leq C\|u_{0}\|_{L_{2}(B^{2}(0))}.$$
 Then, we introduce $\tilde u_{0,a}$ given by Bogovskii's lemma, such that
\begin{align}\label{e.bogoBesov}
\begin{split}
&\nabla\cdot\tilde u_{0,a}=u_0\cdot\nabla\chi,\quad \tilde u_{0,a}=0\ \mbox{on}\ \partial(B_2(0)\setminus B_\frac{3}{2}(0)),\\
&\|\tilde u_{0,a}\|_{L^6(B_{2}(0)\setminus B_{\frac{3}{2}}(0)))}\leq  K_4\|u_0\cdot\nabla\chi\|_{L^2(B_2(0)\setminus B_{\frac{3}{2}}(0))}\leq K_3K_4\|u_0\|_{L^{2}(B_{2}(0))},\\
&\|\tilde u_{0,a}\|_{L^2(B_{2}(0)\setminus B_{\frac{3}{2}}(0))}\leq K_4\|u_0\cdot\nabla\chi\|_{L^2(B_2(0)\setminus B_\frac{3}{2}(0))}\leq K_3K_4\|u_0\|_{L^2(B_{2}(0))},
\end{split}
\end{align}
where $K_4\in(0,\infty)$ is a universal constant. We extend $\tilde u_{0,a}$ by $0$ and let
\begin{equation*}
u_{0,a}=u_0\chi-\tilde u_{0,a}.
\end{equation*}
Clearly,
\begin{equation}\label{L2est}
\|u_{0,a}\|_{L_{2}(\mathbb{R}^3)}\leq (C+K_{3}K_{4})\|u_{0}\|_{L^{2}(\mathbb{R}^3)}.
\end{equation}
Since $\tilde{u}_{0,a}$ has compact support and $L_{3}\hookrightarrow \dot{B}^{-1+\frac{3}{p}}_{p,\infty}$ we have 
\begin{equation}\label{Besovest}
\|u_{0,a}\|_{\dot{B}^{-1+\frac{3}{p}}_{p,\infty}(\mathbb{R}^3)}\leq \gamma'_{univ}+ K_{3}K_{4}\|u_{0}\|_{L^{2}(B_{2}(0))}.
\end{equation}
Using \eqref{cutoffBesovestimate}-\eqref{e.bogoBesov}, together with the heat flow characterisation of Besov spaces gives
\begin{equation}\label{heatflowest}
\sup\limits_{0<t<T} t^{\frac{1}{2}(1-\frac{3}{p})}\|e^{t\Delta} u_{0,a}\|_{L^{p}(\mathbb{R}^3)}\leq C(\chi)\gamma_{univ}+ K_{3}K_{4} MT^{\frac{1}{4}}.
\end{equation}
Using this and Proposition \ref{prop.mildBesov}, there exists $\hat{T}(M,\gamma)$ and  a mild solution $a(\cdot, u_{0,a})$ associated to $u_{0,a}$ on $\mathbb{R}^3 \times (0, \hat{T}(M,\gamma))$. Furthermore, 
$$\sup\limits_{s\in(0,\hat{T})}\big( s^{\frac{1}{2}(1-\frac{3}{p})}\|a(\cdot,s)\|_{L_{p}}+ s^{\frac{1}{2}}\|a(\cdot,s)\|_{L_{\infty}}\big)\leq K_{0}''(p)\sup\limits_{s\in(0,\hat{T})} s^{\frac{1}{2}(1-\frac{3}{p})}\|e^{s\Delta} u_{0,a}\|_{L^{p}(\mathbb{R}^3)}. $$
 Moreover, \eqref{L2est}-\eqref{heatflowest} and Theorem 3.1 of \cite{Barker18} imply that for $t\in (0, \hat{T}(M,\gamma))$:
\begin{equation}\label{estnearinitialtimemildsolution}
\|a(\cdot,t)-e^{t\Delta} u_{0,a}\|_{L^{2}(\mathbb{R}^3)}^2+ \int\limits_{0}^{t}\int\limits_{\mathbb{R}^3} |\nabla(a-e^{t\Delta} u_{0,a})|^2 dxdt'\leq C(M, \hat{T},p) t^{\frac{1}{p-2}}.
\end{equation}
\subsubsection{Local decay estimates near the initial time} Now clearly $v=u-a$ has zero initial data locally on the ball $B_{\frac{3}{2}}(0)$. We next wish to show that for $t\in (0, \min(1,\hat{T}))$ we have
\begin{equation*}\label{Besovdecayinitialtime}
\|v(\cdot,t)\|_{L_{2}(B_{1}(0))}^2+ \int\limits_{0}^{t}\int\limits_{B_{1}(0)} |\nabla v|^2 dxdt' \leq C(M, \hat{T}, \gamma_{univ}) t^{\nu(p)}.
\end{equation*}
for some $\nu(p)>0$. With (\ref{estnearinitialtimemildsolution}) in mind, it is sufficient to show that for $t\in (0, \min( 1, S_{lews}) )$:
\begin{equation*}\label{Besovfluctuationinitialtime}
\|u(\cdot,t)-e^{t\Delta}u_{0,a}\|_{L_{2}(B_{1}(0))}^2+ \int\limits_{0}^{t}\int\limits_{B_{1}(0)} |\nabla (u-e^{t\Delta}u_{0,a})|^2 dxdt' \leq C(M, \hat{T}, \gamma_{univ}) t^{\nu(p)}.
\end{equation*}
In order to show this, we use splitting arguments inspired by the work of C\'{a}lder\'{o}n \cite{calderon1990existence}. The arguments we present here closely follow those presented in \cite{JS13}, \cite{Barker18} and \cite{AlbrittonBarkerBesov2018}.
According to Lemma \ref{splitlem}, we split $u_{0,a}$ into two divergence-free pieces:
\begin{equation}\label{Besovsplitting}
u_{0,a}= \widetilde{u_{0,a}}^{N}+\overline{u_{0,a}}^{N}
\end{equation}
\begin{equation}\label{L2part}
\|\widetilde{u_{0,a}}^{N}\|_{L^{2}}\leq C(M,\gamma)N^{-\gamma_{2}}
\end{equation}
\begin{equation}\label{Besovpart}
\|\overline{u_{0,a}}^{N}\|_{\dot{B}^{-1+\delta_{2}}_{\infty,\infty}}\leq C(M,\gamma) N^{\gamma_{1}}
\end{equation}
\begin{equation}\label{L2subcritical}
\|\overline{u_{0,a}}^{N}\|_{L^{2}}+\|\widetilde{u_{0,a}}^{N}\|_{L^{2}}\leq C(M). 
\end{equation}
Define $u^{N}:= u- e^{t\Delta}\overline{u_{0,a}}^{N}$.
Then
\begin{equation*}\label{VNequation}
\partial_{t}u^{N}-\Delta u^{N}+ u^{N}\cdot \nabla u^{N}+ e^{t\Delta}\overline{u_{0,a}}^{N}\cdot \nabla u^{N}+u^{N}\cdot \nabla e^{t\Delta}\overline{u_{0,a}}^{N}+ \nabla p= -e^{t\Delta}\overline{u_{0,a}}^{N}\cdot\nabla e^{t\Delta}\overline{u_{0,a}}^{N},
\end{equation*}
\begin{equation*}\label{VNdivfree}
\nabla\cdot u^{N}=0
\end{equation*}
\begin{equation*}\label{initialcondition}
u^{N}(x,0)= \widetilde{u_{0,a}}^{N}\,\,\,\,\,\textrm{in}\,\,\,\,\,\,\, B_{\frac{3}{2}}(0).
\end{equation*}
We remark that $p$ is the pressure associated to the original local energy solution $u$. From Proposition \ref{prop.lews} and \eqref{L2subcritical}, we have
\begin{equation}\label{e.aprioriBesov}
\sup_{s\in(0,S_{lews})}\sup_{\bar x\in\R^3}\int\limits_{B_1(\bar x)}\frac{|u^{N}(x,s)|^2}{2}\, dx+\sup_{\bar x\in\R^3}\int\limits_0^{S_{lews}}\int\limits_{B_1(\bar x)}|\nabla u^{N}(x,s)|^2\, dx\, ds\leq C(M).
\end{equation}
Let $\phi\in C^\infty_c(\R^3)$ such that $0\leq\phi\leq 1$, $\supp\phi\subset B_\frac32(0)$, $\phi=1$ on $B_1(0)$, $|\nabla(\phi^2)|\leq K_5$ and $|\Delta(\phi^2)|\leq K_5'$ where $K_5,\ K_5'\in (0,\infty)$ are a universal constants. \\
For $t\in (0,\min(1, S_{lews}))$ we have:
\begin{align*}
&\int\limits_{\R^3}|u^N(\cdot,t)|^2\phi^2 dx+2\int\limits_{0}^t\int\limits_{\R^3}|\nabla u^{N}|^2\phi^2 dxds
\leq \|\widetilde{u_{0,a}^{N}}\|_{L_{2}}^{2}\\&+ \int\limits_{0}^t\int\limits_{\R^3}|u^{N}|^2\Delta(\phi^2) dxds+
\int\limits_0^t\int\limits_{\R^3}|u^{N}|^2u^{N}\cdot\nabla(\phi^2) dxds\\&+\int\limits_0^t\int\limits_{\R^3}2qu^{N}\cdot\nabla(\phi^2) dxds
-\int\limits_0^t\int\limits_{\R^3}(e^{t\Delta}\overline{u_{0,a}}^{N}\cdot\nabla u^{N})\cdot u^{N}\phi^2 dxds\\&+\int\limits_0^t\int\limits_{\R^3}(e^{t\Delta}\overline{u_{0,a}}^{N}\otimes u^{N}):\nabla u^{N} \phi^2 dxds
+\int\limits_0^t\int\limits_{\R^3}(e^{t\Delta}\overline{u_{0,a}}^{N}\otimes u^{N}):u^{N}\otimes\nabla(\phi^2) dxds\\&+\int\limits_0^t\int\limits_{\R^3}(e^{t\Delta}\overline{u_{0,a}}^{N}\otimes e^{t\Delta}\overline{u_{0,a}}^{N}):\nabla u^{N} \phi^2 dxds\\
&+\int\limits_0^t\int\limits_{\R^3}(e^{t\Delta}\overline{u_{0,a}}^{N}\otimes e^{t\Delta}\overline{u_{0,a}}^{N}):u^{N}\otimes\nabla(\phi^2) dxds \\ &= I_{0}+I_1+\ldots\ I_8,
\end{align*}
Using \eqref{L2part}, we have $I_{0}\leq C(M, \gamma) N^{-2\gamma_{2}}.$
Using \eqref{e.aprioriBesov} and the same arguments as in Section \ref{sec.proof} gives
\begin{equation*}\label{I1-3Besov}
|I_{1}|+|I_{2}|+|I_{3}|\leq C(M)t^{\frac{1}{10}}.
\end{equation*}
Furthermore, using \eqref{e.aprioriBesov} and \eqref{Besovpart} we obtain
\begin{equation*}\label{I6I5}
|I_{4}|+|I_{5}|+|I_{6}|\leq C(M,\gamma, \delta_{2}) N^{\gamma_{1}}t^{\frac{\delta_{2}}{2}}.
\end{equation*}
Next, we may use \eqref{Besovpart}-\eqref{L2subcritical} to see that
$$\int\limits_{0}^{t}\int\limits_{\mathbb{R}^3} |e^{t\Delta}\overline{u_{0,a}}^{N}|^4 dxdt'\leq C(M,\gamma,\delta_{2}) N^{2\gamma_{1}} t^{\delta_{2}}. $$
This may be used with \eqref{e.aprioriBesov} to show
\begin{equation*}\label{I7I8Besov}
|I_{7}|+|I_{8}|\leq C(M,\gamma, \delta_{2})N^{2\gamma_{1}}t^{\delta_{2}}.
\end{equation*}
Thus for $t\in (0,\min(1, S_{lews}))$ we have
\begin{multline*}
\|u^{N}(\cdot,t)\|_{L^{2}(B_{1}(0))}+\int\limits_{0}^{t}\int\limits_{B_{1}(0)} |\nabla u^{N}(x,t')|^2 dxdt'\\
\leq  C(M,\gamma, \delta_{2})(N^{-2\gamma_{2}}+N^{\gamma_{1}}+N^{2\gamma_{1}})t^{\min(\frac{1}{10}, \frac{\delta_{2}}{2})}.
\end{multline*}
Noting that $u-e^{t\Delta} u_{0,a}= u^{N}- \widetilde{u_{0,a}}^{N}$, we thus obtain for $t\in (0,\min(1, S_{lews}))$ and $N\in (0,\infty)$ that 
\begin{multline*}
\|u(\cdot,t)-e^{t\Delta} u_{0,a}\|_{L^{2}(B_{1}(0))}+\int\limits_{0}^{t}\int\limits_{B_{1}(0)} |\nabla (u-e^{t\Delta}u_{0,a})|^2 dxdt'\\
\leq  C(M,\gamma, \delta_{2})(N^{-2\gamma_{2}}+N^{\gamma_{1}}+N^{2\gamma_{1}})t^{\min(\frac{1}{10}, \frac{\delta_{2}}{2})}.
\end{multline*}
Choosing $N= t^{-\beta}$, where $\beta>0$ is sufficiently small, then yields the desired estimate \eqref{Besovfluctuationinitialtime}.

\subsection{Besov initial data: Section \ref{sec.morrey}}
In this section we give the adjustments needed to prove Theorem \ref{theo.morrey} in the case of a drift $a$, which rather than satisfying the global $L^5(Q_1(0,0))$ bound \eqref{e.epmorreya}, just satisfies 
\begin{equation*}
\sup_{s\in (-1,0)}|s-t_0|^\frac15\|a(\cdot,s)\|_{L^5(B_1(0))}<\infty
\end{equation*} 
and small, for a fixed $t_0\in [-1,0]$. This extension is needed to deal with the case of locally Besov initial data $\dot B^{-1+\frac3p}_{p,\infty}$, with $p=5$, for which the mild solution just satisfies \eqref{e.boundmildBesov}. We actually prove the following theorem which allows to handle any $p\in(3,\infty)$.

\begin{theorem}\label{theo.morreygen}
Let $t_0\in[-1,0]$ and $\eta\in(0,1)$ be fixed. For all $\delta\in(0,3)$, there exists $C_*(\delta)\in (0,\infty)$, 
for all $E\in(0,\infty)$, 
there exists $\ep_*(\delta,\eta,E)\in (0,\infty)$, for all $a$ such that
\begin{equation*}
\sup_{s\in(-1,0)}|s-t_0|^{\frac12}\|a(\cdot,s)\|_{L^\infty(B_1(0))}<\infty
\end{equation*}
 and all local suitable solution $v$ to \eqref{e.pertns} in $Q_1(0,0)$ such that\footnote{By definition $(\cdot)_+:=\max(0,\cdot)$.}
\begin{align}
\int\limits_{B_1(0)}|v(x,s)|^2dx+\int\limits_{-1}^{s}\int\limits_{B_1(0)}|\nabla v|^2dxds\leq\ & E(s-t_0)_+^\eta,\quad\forall s\in(-1,0),\label{e.locenEgen}\\
\int\limits_{-1}^s\int\limits_{B_1(0)}|q|^\frac32dxd\hat s\leq\ &E(s-t_0)_+^{\frac34\eta},\quad\forall s\in(-1,0),\label{e.locenPgen}
\end{align}
the conditions\footnote{Notice that \eqref{e.epmorreyagen} can be also replaced by the weaker assumption
\begin{equation*}
\sup_{s\in(-1,0)}(s-t_0)_+^{\frac12}\|a(\cdot,s)\|_{L^\infty(B_1(0))}
\leq \ep_*.
\end{equation*}}
\begin{equation}\label{e.epmorreyagen}
\sup_{s\in(-1,0)}|s-t_0|^{\frac12}\|a(\cdot,s)\|_{L^\infty(B_1(0))}
\leq \ep_*
\end{equation}
and
\begin{equation}\label{e.epmorreygen}
\int\limits_{Q_1(0,0)}|v|^3+|q|^\frac32dxds\leq\ep_*
\end{equation}
imply that for all $(\bar x,t)\in \bar Q_{1/2}(0,0)$, for all $r\in(0,\frac14]$,
\begin{equation}\label{e.morreyboundgen}
\intbar_{Q_r(\bar x,t)}|v|^3dxds\leq C_*\ep_*^\frac23r^{-\delta}.
\end{equation}
\end{theorem}

We note that \eqref{e.locenEgen} implies in particular that for all $s\in(-1,0)$, $v(\cdot,s)=0$ a.e. on $B_1(0)$. As was emphasized just below Theorem \ref{theo.morrey}, the constant $C_*$ only depends on $\delta$, because it arises when going from scale $r_n$ to $r\in (r_{n+1},r_n)$.

The proof goes through using the same general scheme as in Section \ref{sec.morrey}. The main difficulty is that the bound \eqref{e.epmorreyagen} does not imply $a\in L^\frac{2}{1-\frac3p}(-1,0;L^p(B_1(0)))$. Hence estimates on the term 
\begin{equation*}
I_4=2\int\limits_{Q_\frac12(\bar x,t)}|a||v||\nabla v||\phi_n|dxds
\end{equation*}
carried out in Section \ref{sec.morrey} do not work as such any longer. One possible way out is to use that in our application of Theorem \ref{theo.morrey} to Theorem \ref{theo.main}, we have smallness of the local energy of the perturbation $v$ in short time. This is estimate \eqref{Besovdecayinitialtime}. This smallness is expressed in the assumption \eqref{e.locenE}, which allows to remove the singularity due to \eqref{e.epmorreyagen}. Consequently, there are two main modifications to the argument in Section \ref{sec.morrey}. The first modification is on the bounds \eqref{e.ak} and \eqref{e.bk} which are iterated. The second modification is on Lemma \ref{lem.pressure} for the pressure.

Let $(\bar x,t)\in \bar Q_{1/2}(0,0)$ be fixed for the rest of this section. For all $n\in\N$, we let $r_n:=2^{-n}$. Our aim is to propagate for $k\geq 2$ the following three bounds
\begin{align}
\tag{$A_k'$}\label{e.akp}
\frac1{r_k^2}\int\limits_{t-r_k^2}^s\int\limits_{B_{r_k}(\bar x)}|v|^3dxd\hat s\leq\ &\frac12(s-t_0)_+^{\frac32\eta'}\ep_*^\frac23r_k^{3-\delta},\quad t-r_k^2<s<t,\\
\tag{$A_k''$}\label{e.akpp}\frac1{r_k^{\frac{1+\delta}2}}\int\limits_{t-r_k^2}^s\int\limits_{B_{r_k}(\bar x)}|q-(q)_{r_k}(\hat s)|^\frac32dxd\hat s\leq\ &\frac12(s-t_0)_+^{\frac34\eta'}\ep_*^\frac23r_k^{3-\delta},\quad t-r_k^2<s<t,
\end{align}
and 
\begin{align}
\tag{$B_k'$}\label{e.bkp}
\int\limits_{B_{r_k}(\bar x)}|v(x,s)|^2dx+\int\limits_{t-r_k^2}^s\int\limits_{B_{r_k}(\bar x)}|\nabla v|^2dxds\leq\ & C_B(s-t_0)_+^{\eta'}\ep_*^\frac23r_k^{3-\frac23\delta},\quad t-r_k^2<s<t,
\end{align}
where 
\begin{equation*}
(q)_{r_k}(s):=\intbar_{B_{r_k}(\bar x)}q(x,s)dx,
\end{equation*}
for $\eta'=\frac\eta6\in(0,\frac16)$ and constants $\ep_*(\delta,\eta,E),\, C_B(\delta,\eta)\in (0,\infty)$ to be chosen. Notice that the power $\frac34\eta'$ in \eqref{e.akpp} is worse than the corresponding power in \eqref{e.akp}. This fact appears in Step 4 of the proof of Theorem \ref{theo.morreygen}. It is due to the fourth term in the right hand side of \eqref{e.estpressbis} below.

We also need the following modification of Lemma \ref{lem.pressure}.

\begin{lemma}[Pressure estimate]\label{lem.pressurebis}
There exists a constant $C_2'\in(0,\infty)$ such that for all $\rho\in(0,\infty)$, for all $a$ such that\begin{equation}\label{e.defMa}
M_a:=\sup_{s\in(-1,0)}|s-t_0|^{\frac12}\|a(\cdot,s)\|_{L^\infty(B_1(0))}<\infty
\end{equation}
for all weak solution $q\in L^\frac32(Q_\rho(0,0))$ to 
\begin{equation*}
-\Delta q=\nabla\cdot\nabla\cdot(v\otimes v)+\nabla\cdot\nabla\cdot(a\otimes v)+\nabla\cdot\nabla\cdot(v\otimes a)\quad\mbox{in}\quad Q_\rho(0,0),
\end{equation*}
we have 
\begin{align}
\label{e.estpressbis}
\begin{split}
&r^{-\frac{1+\delta}2}\int\limits_{-r^2}^s\int\limits_{B_{r}(0)}|q-(q)_r(\hat s)|^\frac32dxd\hat s\\
\leq\ & C_2'r^{-\frac{1+\delta}2}\int\limits_{-r^2}^s\int\limits_{B_{2r}(0)}|v|^3dxd\hat s+C_2'r^{\frac34-\frac{\delta}2}M_a^\frac32\Bigg(\int\limits_{-r^2}^s\frac1{|\hat s-t_0|}\int\limits_{B_{2r}(0)}|v(x,\hat s)|^2dxd\hat s\Bigg)^\frac34\\
&+C_2'r^{6-\frac\delta2}\Bigg(\sup_{-r^2<\hat s<s}\int\limits_{2r<|x|<\rho}\frac{|v(x,\hat s)|^2}{|x|^4}dx\Bigg)^\frac32\\
&+C_2'r^{4-\frac\delta2}M_a^\frac32\int\limits_{-r^2}^s\frac1{|\hat s-t_0|^\frac34}\Bigg(\int\limits_{2r<|x|<\rho}\frac{|v(x,\hat s)|}{|x|^4}dx\Bigg)^\frac32d\hat s\\
&+C_2'r^{4-\frac\delta2}\rho^{-\frac92}\int\limits_{-r^2}^s\int\limits_{B_\rho(0)}|v|^3+|q|^\frac32dxd\hat s\\
&+C_2'r^{4-\frac\delta2}\rho^{-\frac{15}{4}}M_a^\frac32\int\limits_{-r^2}^s\frac1{|\hat s-t_0|^\frac34}\Bigg(\int\limits_{B_\rho(0)\setminus B_{\frac\rho2}(0)}|v(x,\hat s)|^2dx\Bigg)^\frac3{4}d\hat s,
\end{split}
\end{align}
for all $s\in (-r^2,0)$, for all $0<r\leq\rho/2$.
\end{lemma}

\begin{proof}[Sketch of the proof of Theorem \ref{theo.morreygen}]
\emph{In the whole proof, we define $M_a$ as in \eqref{e.defMa}.} Notice that by assumption \eqref{e.epmorreyagen}, $M_a\leq\ep_*$. 
Let us sketch the main differences with respect to the proof of Theorem \ref{theo.morrey} in Section \ref{sec.morrey}. We focus on the case when $(\bar x,t)=(0,0)$, but the argument for general $(\bar x,t)\in Q_{\frac12}(0,0)$ follows along the same lines.

\noindent\underline{\bf Step 1: \eqref{e.akp} and \eqref{e.akpp} for $k=2$.} This step is slightly different from the analogous step in Section \ref{sec.morrey}. Indeed, assumption \eqref{e.epmorreygen} does not imply any rate of decay near the time $t_0$. Therefore, we have to combine \eqref{e.epmorreygen}, to get the smallness with respect to $\ep_*$, with \eqref{e.locenEgen} or \eqref{e.locenPgen}, to get the decay rate in time. To do so, one has to give up a bit of the power $\eta$. Indeed, instead of $(s-t_0)_+^\eta$, the decay rate in \eqref{e.akp} is $(s-t_0)_+^\frac\eta6$. We have
\begin{align*}
\frac1{r_2^2}\int\limits_{-r_2^2}^s\int\limits_{B_{r_2}(0)}|v|^3dxd\hat s\leq\ &\frac1{r_2^2}\Bigg(\int\limits_{-r_2^2}^s\int\limits_{B_{r_2}(0)}|v|^3dxd\hat s\Bigg)^\frac16\Bigg(\int\limits_{-r_2^2}^s\int\limits_{B_{r_2}(0)}|v|^3dxd\hat s\Bigg)^\frac56\\
\leq\ &C\big(E^\frac32(s-t_0)_+^{\frac32\eta}\big)^\frac16\ep_*^\frac56\\
\leq\ &C(\delta)E^\frac14(s-t_0)_+^{\frac32\eta'}\ep_*^\frac56r_2^{3-\delta}.
\end{align*}
The estimate for the pressure using \eqref{e.locenPgen} is similar.

\noindent\underline{\bf Step 2: \eqref{e.akp} and \eqref{e.akpp} for $k=2$ imply \eqref{e.bkp} for $k=2$.} We do not give the details for this step. Similar calculations are done below in Step 3. Notice that the terms $I_4$ and $I_5$ have to be estimated using \eqref{e.locenEgen} and \eqref{e.akp} for $k=2$. The smallness of $a$ given by \eqref{e.epmorreyagen} enables to absorb some constants by choosing $\ep_*$ small enough.

\noindent\underline{\bf Step 3: \eqref{e.akp}, \eqref{e.akpp} and \eqref{e.bk} for $2\leq k\leq n$ imply \eqref{e.bkp} for $k=n+1$.} Thanks to the local energy inequality \eqref{e.lei}, we have for all $s\in(-r_n^2,0)$,
\begin{align*}
&C_1^{-1}r_n^{-1}\int\limits_{B_{r_n}(0)}|v(x,s)|^2dx+C_1^{-1}r_n^{-1}\int\limits_{-r_n^2}^s\int\limits_{B_{r_n}(0)}|\nabla v|^2dxds\\
\leq\ &
C_1r_n^2\int\limits_{-(\frac12)^2}^s\int\limits_{B_{1/2}(0)}|v|^2dxd\hat s+\int\limits_{-(\frac12)^2}^s\int\limits_{B_{1/2}(0)}|v|^3|\nabla\phi_n|dxd\hat s\\
&+2\left|\int\limits_{-(\frac12)^2}^s\int\limits_{B_{1/2}(0)}v\cdot\nabla\phi_n qdxd\hat s\right|+2\int\limits_{-(\frac12)^2}^s\int\limits_{B_{1/2}(0)}|a||v||\nabla v||\phi_n|dxd\hat s\\
&+\int\limits_{-(\frac12)^2}^s\int\limits_{B_{1/2}(0)}|a||v|^2|\nabla\phi_n|dxd\hat s\\
=\ &I_1'+\ldots\, I_5'.
\end{align*}
For $I_1'$, we have
\begin{align*}
|I_1'|\leq\ & C_1(s-t_0)^\frac13_+\ep_*^\frac23r_n^{2-\frac23\delta}.
\end{align*}
The term $I_2'$ is immediate following the estimates of Section \ref{sec.morrey}. Let us write some details for $I_3'$. We have
\begin{align*}
|I_3'|\leq\  &Cr_n^2\sum_{k=2}^nr_{k-1}^{-4}\Bigg(\int\limits_{-r_{k-1}^2}^s\int\limits_{B_{r_{k-1}}(0)}|q-(q)_{r_{k-1}}(\hat s)|^\frac32dxd\hat s\Bigg)^\frac23\Bigg(\int\limits_{-r_{k-1}^2}^s\int\limits_{B_{r_{k-1}}(0)}|v|^3dxd\hat s\Bigg)^\frac13\\
&+Cr_n^{-2}\Bigg(\int\limits_{-r_{n}^2}^s\int\limits_{B_{r_{n}}(0)}|q-(q)_{r_{n}}(\hat s)|^\frac32dxd\hat s\Bigg)^\frac23\Bigg(\int\limits_{-r_{n}^2}^s\int\limits_{B_{r_{n}}(0)}|v|^3dxd\hat s\Bigg)^\frac13\\
\leq\ &C(\delta)(\ep_*^\frac23)^\frac23r_n^{-2}\big(r_n^{\frac72-\frac\delta2}(s-t_0)_+^{\frac34\eta'}\big)^\frac23(\ep_*^\frac23)^\frac13\big(r_n^{5-\delta}(s-t_0)_+^{\frac32\eta'}\big)^\frac13=C(\delta)(s-t_0)_+^{\eta'}\ep_*^\frac23r_n^{2-\frac23\delta}.
\end{align*}
Some changes are necessary to deal with $I_4'$ and $I_5'$. For $I_4'$, 
\begin{align*}
|I_4'|\leq\ &Cr_n^2\sum_{k=1}^nr_{k+1}^{-3}\int\limits_{-r_k^2}^s\int\limits_{B_{r_k}(0)}|a||v||\nabla v|dxd\hat s\\
\leq\ &Cr_n^2M_a\sum_{k=1}^nr_{k+1}^{-3}\int\limits_{-r_k^2}^s\frac1{|\hat s-t_0|^\frac12}\Bigg(\int\limits_{B_{r_k}(0)}|v(\cdot,\hat s)|^2dx\Bigg)^\frac12\Bigg(\int\limits_{B_{r_k}(0)}|\nabla v(\cdot,\hat s)|^2dx\Bigg)^\frac12d\hat s\\
\leq\ &Cr_n^2M_a\ep_*^\frac13\sum_{k=1}^nr_{k+1}^{-3}r_k^{\frac32-\frac\delta 3}\Bigg(\int\limits_{t_0}^s\frac1{(\hat s-t_0)^{1-\eta'}}d\hat s\Bigg)^\frac12\Bigg(\int\limits_{-r_k^2}^s\int\limits_{B_{r_k}(0)}|\nabla v(\cdot,\hat s)|^2dxd\hat s\Bigg)^\frac12\\
\leq\ & C(\eta',\delta)(s-t_0)^{\eta'}\ep_*^{1+\frac23}r_n^{2-\frac23\delta},
\end{align*}
using the fact that $M_a\leq\ep_*$ by assumption. Finally, 
\begin{align*}
|I_5'|\leq\ &Cr_n^2\sum_{k=1}^nr_{k+1}^{-4}\int\limits_{-r_k^2}^s\int\limits_{B_{r_k}(0)}|a||v|^2dxd\hat s\\
\leq\ &Cr_n^2M_a\sum_{k=1}^nr_{k+1}^{-4}\int\limits_{-r_k^2}^s\frac1{|\hat s-t_0|^\frac12}\int\limits_{B_{r_k}(0)}|v(\cdot,\hat s)|^2dxd\hat s\\
\leq\ &C(\eta',\delta)(s-t_0)_+^{\frac12+\eta'}\ep_*^{1+\frac23}r_n^{2-\frac23\delta}.
\end{align*}
This concludes Step 3.

\noindent\underline{\bf Step 4: \eqref{e.bkp} for $2\leq k\leq n+1$ implies \eqref{e.akp} and \eqref{e.akpp} for $k=n+1$.} We first prove the estimate \eqref{e.akp}. We have
\begin{align*}
\int\limits_{-r_{n+1}^2}^s\int\limits_{B_{r_{n+1}}(0)}|v|^3dxd\hat s\leq\ &r_{n+1}^{-\frac32}\int\limits_{-r_{n+1}^2}^s\Bigg(\int\limits_{B_{r_{n+1}}(0)}|v|^2dx\Bigg)^\frac32d\hat s\\
&+\int\limits_{-r_{n+1}^2}^s\Bigg(\int\limits_{B_{r_{n+1}}(0)}|v|^2dx\Bigg)^\frac34\Bigg(\int\limits_{B_{r_{n+1}}(0)}|\nabla v|^2dx\Bigg)^\frac34d\hat s\\
\leq\ &r_{n+1}^{-\frac32}\int_{-r_{n+1}^2}^s\big(C_B\ep_*^\frac23(\hat s-t_0)_+^{\eta'}r_{n+1}^{3-\frac23\delta}\big)^\frac32d\hat s\\
&+\Bigg(\int\limits_{-r_{n+1}^2}^s\big(C_B\ep_*^\frac23(\hat s-t_0)_+^{\eta'} r_{n+1}^{3-\frac23\delta}\big)^3d\hat s\Bigg)^\frac14\Bigg(\int\limits_{-r_{n+1}^2}^s\int\limits_{B_{r_{n+1}}(0)}|\nabla v|^2dxd\hat s\Bigg)^\frac34\\
\leq\ &2C_B^\frac32\ep_*(s-t_0)^{\frac32\eta'}r_{n+1}^{5-\delta},
\end{align*}
which proves \eqref{e.akp} for $k=n+1$ by choosing $\ep_*$ sufficiently small. 
Let us prove the estimate for the pressure using the bound of Lemma \ref{lem.pressurebis}. We take $r=r_{n+1}$ and $\rho=\frac14$. We have 
\begin{align*}
\begin{split}
&r_{n+1}^{-\frac{1+\delta}2}\int\limits_{-r_{n+1}^2}^s\int\limits_{B_{r_{n+1}}(0)}|q-(q)_{r_{n+1}}(\hat s)|^\frac32dxd\hat s\\
\leq\ & C_2'r_{n+1}^{-\frac{1+\delta}2}\int\limits_{-r_{n+1}^2}^s\int\limits_{B_{r_{n}}(0)}|v|^3dxd\hat s+C_2'
r_{n+1}^{\frac34-\frac{\delta}2}M_a^\frac32\Bigg(\int\limits_{-r_{n+1}^2}^s\frac1{|\hat s-t_0|}\int\limits_{B_{{r_{n}}}(0)}|v(x,\hat s)|^2dxd\hat s\Bigg)^\frac34\\
&+C_2'r_{n+1}^{6-\frac\delta2}\Bigg(\sup_{-r_{n+1}^2<\hat s<s}\int\limits_{{r_{n}}<|x|<\frac14}\frac{|v(x,\hat s)|^2}{|x|^4}dx\Bigg)^\frac32\\
&+C_2'r_{n+1}^{4-\frac\delta2}M_a^\frac32\int\limits_{-r_{n+1}^2}^s\frac1{|\hat s-t_0|^\frac34}\Bigg(\int\limits_{r_{n}<|x|<\frac14}\frac{|v(x,\hat s)|}{|x|^4}dx\Bigg)^\frac32d\hat s\\
&+2^9C_2'r_{n+1}^{4-\frac\delta2}\int\limits_{-r_{n+1}^2}^s\int\limits_{B_{\frac14}(0)}|v|^3+|q|^\frac32dxd\hat s\\
&+2^{\frac{15}{2}}C_2'r_{n+1}^{4-\frac\delta2}M_a^\frac32\int\limits_{-r_{n+1}^2}^s\frac1{|\hat s-t_0|^\frac34}\Bigg(\int\limits_{B_{\frac14}(0)\setminus B_{\frac18}(0)}|v(x,\hat s)|^2dx\Bigg)^\frac3{4}d\hat s\\
=\ &J_1'+\ldots\ J_6',
\end{split}
\end{align*}
for all $s\in (-r_{n+1}^2,0)$. We concentrate on the estimates for $J_2'$, $J_3'$ and $J_4'$. The estimate of $J_1'$ is similar to the one just done above. The estimates of $J_5'$ and $J_6'$ do not pose any additional difficulty. For $J_2'$, we have
\begin{align*}
|J_2'|\leq\ &C_2'r_{n+1}^{\frac34-\frac{\delta}2}\ep_*^\frac32C_B^\frac34\ep_*^\frac12r_n^{\frac94-\frac\delta2}\Bigg(\int\limits_{t_0}^s\frac1{(\hat s-t_0)^{1-\eta'}}d\hat s\Bigg)^\frac34\leq C(\eta')C_B^\frac34\ep_*^2(\hat s-t_0)_+^{\frac34\eta'}r_{n+1}^{3-\delta}.
\end{align*} 
For $J_3'$, the estimate is very close to the bound for $J_3$ in Section \ref{sec.morrey}. We also split the integral into rings. This yields
\begin{align*}
|J_3'|\leq\ &C_2'r_{n+1}^{6-\frac\delta2}\ep_*C_B^\frac23C(\delta)(s-t_0)_+^{\frac32\eta'}r_{n+1}^{-\frac32-\delta}\leq C_2'C_B^\frac23C(\delta)\ep_*(s-t_0)_+^{\frac32\eta'}r_{n+1}^{3-\delta}.
\end{align*}
Finally for $J_4'$ splitting again into rings leads to
\begin{align*}
|J_4'|\leq\ & C_2'r_{n+1}^{4-\frac\delta2}\ep_*^\frac32C_B^\frac34C(\delta)\ep_*^\frac12r_{n+1}^{-\frac32-\frac\delta2}\int\limits_{-r_{n+1}^2}^s\frac{(\hat s-t_0)_+^{\frac34\eta'}}{|s-t_0|^\frac34}d\hat s\\
\leq\ &C_2'r_{n+1}^{4-\frac\delta2}\ep_*^\frac32C_B^\frac34C(\delta)\ep_*^\frac12r_{n+1}^{-\frac32-\frac\delta2}(s-t_0)_+^{\frac34\eta'}r_{n+1}^\frac12\\
\leq\ &C_2'C_B^\frac34C(\delta)\ep_*(s-t_0)_+^{\frac34\eta'}r_n^{3-\delta}.
\end{align*} 
Hence the estimate \eqref{e.akpp} follows for $\ep_*$ sufficiently small. 
\end{proof}

\subsection*{Funding and conflict of interest.} 
The second author is partially supported by the project BORDS (ANR-16-CE40-0027-01) operated by the French National Research Agency (ANR). 
The second author also acknow\-ledges financial support from the IDEX of the University of Bordeaux for the BOLIDE project.  The authors declare that they have no conflict of interest. 

\subsection*{Acknowledgement}
Both authors warmly thank the OxPDE centre, where this work started. The second author would also like to thank Yasunori Maekawa and Hideyuki Miura for sti\-mulating discussions about concentration for blowing-up solutions of the Navier-Stokes equations.

\small
\bibliographystyle{abbrv}
\bibliography{concentration.bib}

\end{document}